\documentclass[a4paper,12pt]{article}
\usepackage{latexsym,amsmath,amssymb,amsfonts,amsthm,dsfont,enumitem,xcolor,natbib,bbm,threeparttable}
\usepackage[normalem]{ulem}

\usepackage[colorlinks,linkcolor=red,citecolor=blue,urlcolor=blue,breaklinks]{hyperref}

\newtheorem{thm}{Theorem}
\newtheorem{prop}{Proposition}
\newtheorem{lemma}{Lemma}
\newtheorem{cor}{Corollary}

\newtheorem*{remark}{Remark}

\DeclareMathOperator*{\argmin}{argmin}
\DeclareMathOperator*{\argmax}{argmax}

\RequirePackage{lineno}

\begin{document}
 
\title{Isotonic regression in general dimensions}

\author{Qiyang Han\footnote{University of Washington.  Email: royhan@uw.edu}, Tengyao Wang\footnote{University of Cambridge.  Email: t.wang@statslab.cam.ac.uk}, \\ Sabyasachi Chatterjee\footnote{University of Chicago and University of Illinois at Urbana-Champaign.  Email: sc1706@illinois.edu} and Richard J. Samworth\footnote{University of Cambridge.  Email: r.samworth@statslab.cam.ac.uk}}

\maketitle

\begin{abstract}
We study the least squares regression function estimator over the class of real-valued functions on $[0,1]^d$ that are increasing in each coordinate.  For uniformly bounded signals and with a fixed, cubic lattice design, we establish that the estimator achieves the minimax rate of order $n^{-\min\{2/(d+2),1/d\}}$ in the empirical $L_2$ loss, up to poly-logarithmic factors.  Further, we prove a sharp oracle inequality, which reveals in particular that when the true regression function is piecewise constant on $k$ hyperrectangles, the least squares estimator enjoys a faster, adaptive rate of convergence of $(k/n)^{\min(1,2/d)}$, again up to poly-logarithmic factors.  Previous results are confined to the case $d \leq 2$.  Finally, we establish corresponding bounds (which are new even in the case $d=2$) in the more challenging random design setting.  There are two surprising features of these results: first, they demonstrate that it is possible for a global empirical risk minimisation procedure to be rate optimal up to poly-logarithmic factors even when the corresponding entropy integral for the function class diverges rapidly; second, they indicate that the adaptation rate for shape-constrained estimators can be strictly worse than the parametric rate.
\end{abstract}

\section{Introduction}
\label{Sec:Intro}

Isotonic regression is perhaps the simplest form of shape-constrained estimation problem, and has wide applications in a number of fields.  For instance, in medicine, the expression of a leukaemia antigen has been modelled as a monotone function of white blood cell count and DNA index \citep{SchellSingh1997}, while in education, isotonic regression has been used to investigate the dependence of college grade point average on high school ranking and standardised test results \citep{DykstraRobertson1982}.  It is often generally accepted that genetic effects on phenotypes such as height, fitness or disease are monotone \citep{Manietal2007,RLA2009,LRS2012}, but additive structures have been found to be inadequate in several instances \citep{Shaoetal2008,Goldstein2009,Eichleretal2010}.  Alternative simplifying interaction structures have also been considered, including those based on products \citep{ElenaLenski1997}, logarithms \citep{SanjuanElena2006} and minima \citep{Tongetal2001}, but the form of genetic interaction between factors is not always clear and may vary between phenotypes \citep{LRS2012}.  

A simple class of isotonic functions, which includes all of the above structures as special cases, is the class of block increasing functions
\[
\mathcal{F}_d := \bigl\{\text{$f: [0,1]^d \to \mathbb{R}$, $f(x_1,\ldots,x_d) \leq f(x'_1,\ldots,x'_d)$  when $x_j\leq x'_j$ for $j = 1,\ldots,d$}\bigr\}.
\]
In this paper, we suppose that we observe data $(X_1,Y_1),\ldots,(X_n,Y_n)$, with $n \geq 2$, satisfying
\begin{equation}
\label{Eq:Model}
  Y_i = f_0(X_i) + \epsilon_i,\quad i=1,\ldots,n,
\end{equation}
where $f_0:[0,1]^d \rightarrow \mathbb{R}$ is Borel measurable, $\epsilon_1,\ldots,\epsilon_n$ are independent $N(0,1)$ noise, and the covariates $X_1,\ldots,X_n$, which take values in the set $[0,1]^d$, can either be fixed or random.  Our goal is to study the performance of the least squares isotonic regression estimator $\hat{f}_n \in \argmin_{f\in\mathcal{F}_d} \sum_{i=1}^n \{Y_i - f(X_i)\}^2$ in terms of its empirical risk
\begin{equation}\label{Eq:Loss}
  R(\hat{f}_n, f_0) := \mathbb{E} \biggl[\frac{1}{n}\sum_{i=1}^n \{\hat{f}_n(X_i) - f_0(X_i)\}^2\biggr].
\end{equation}
Note that this loss function only considers the errors made at the design points $X_1,\ldots,X_n$, and these design points naturally induce a directed acyclic graph $G_X = (V(G_X),E(G_X))$ with $V(G_X) = \{1,\ldots,n\}$ and $E(G_X) = \{(i,i') : (X_{i})_{j}\leq (X_{i'})_{j}\; \forall\, j=1,\ldots,d\}$.  It is therefore natural to restate the problem in terms of isotonic vector estimation on directed acyclic graphs. Recall that given a directed acyclic graph $G = (V(G),E(G))$, we may define a partially ordered set $(V(G),\leq)$, where $u\leq v$ if and only if there exists a directed path from $u$ to $v$. We define the class of isotonic vectors on $G$ by
\[
 \mathcal{M}(G) := \{\theta \in \mathbb{R}^{V(G)} : \theta_u\leq \theta_v \text{ for all $u\leq v$}\}. 
\]
Hence, for a signal vector $\theta_0 = ((\theta_0)_i)_{i=1}^n := (f_0(X_i))_{i=1}^n \in \mathcal{M}(G_X)$, the least squares estimator $\hat\theta_n = ((\hat\theta_n)_i)_{i=1}^n := (\hat{f}_n(X_i))_{i=1}^n$ can be seen as the projection of $(Y_i)_{i=1}^n$ onto the polyhedral convex cone $\mathcal{M}(G_X)$. Such a geometric interpretation means that least squares estimators for isotonic regression, in general dimensions or on generic directed acyclic graphs, can be efficiently computed using convex optimisation algorithms (see, e.g., \citet{Dykstra1983, KRS2015, Stout2015}).

In the special case where $d = 1$, model \eqref{Eq:Model} reduces to the univariate isotonic regression problem that has a long history \citep[e.g.][]{Brunk1955, vanEeden1958, BBBB1972, vandeGeer1990, vandeGeer1993, Donoho1991, BirgeMassart1993, MeyerWoodroofe2000, Durot2007, Durot2008, YangBarber2017}.  See \citet{GroeneboomJongbloed2014} for a general introduction.  Since the risk only depends on the ordering of the design points in the univariate case, fixed and random designs are equivalent for $d = 1$ under the empirical risk function (\ref{Eq:Loss}).  It is customary to write $R(\hat\theta_n,\theta_0)$ in place of $R(\hat f_n, f_0)$ for model~\eqref{Eq:Model} with fixed design points.  When $(\theta_0)_1\leq\cdots\leq(\theta_0)_n$ (i.e.\ $X_1\leq \cdots\leq X_n$), \citet{Zhang2002} proved that there exists a universal constant $C > 0$ such that
\begin{equation}
\label{Eq:WorstCase1d}
  R(\hat\theta_n, \theta_0) \leq C\biggl\{\biggl(\frac{(\theta_0)_n-(\theta_0)_1}{n}\biggr)^{2/3} + \frac{\log n}{n}\biggr\},
\end{equation}
which shows in particular that the risk of the least squares estimator is no worse than $O(n^{-2/3})$ for signals $\theta_0$ of bounded uniform norm.  In recent years, there has been considerable interest and progress in studying the automatic rate-adaptation phenomenon of shape-constrained estimators.  This line of study was pioneered by \citet{Zhang2002} in the context of univariate isotonic regression, followed by \citet{ChatterjeeGuntuboyinaSen2015} and most recently \citet{Bellec2017}, who proved that 
\begin{equation}
\label{Eq:Oracle1d}
  R(\hat\theta_n, \theta_0) \leq \inf_{\theta\in\mathcal{M}(G_X)}\biggl\{\frac{\|\theta-\theta_0\|_2^2}{n} + \frac{k(\theta)}{n}\log \biggl(\frac{en}{k(\theta)}\biggr)\biggr\},
\end{equation}
where $k(\theta)$ is the number of constant pieces in the isotonic vector $\theta$.  The inequality~\eqref{Eq:Oracle1d} is often called a \emph{sharp oracle inequality}, with the sharpness referring to the fact that the approximation error term $n^{-1}\|\theta_0-\theta\|_2^2$ has leading constant 1.  The bound~\eqref{Eq:Oracle1d} shows nearly parametric adaptation of the least squares estimator in univariate isotonic regression when the underlying signal has a bounded number of constant pieces.  Other examples of adaptation in univariate shape-constrained problems include the maximum likelihood estimator of a log-concave density \citep{KGS2017}, and the least squares estimator in unimodal regression \citep{ChatterjeeLafferty2017}.

Much less is known about the rate of convergence of the least squares estimator in the model~\eqref{Eq:Model}, or indeed the adaptation phenomenon in shape-restricted problems more generally, in multivariate settings. The only work of which we are aware in the isotonic regression case is \citet{ChatterjeeGuntuboyinaSen2017}, which deals with the fixed, lattice design case when $d = 2$.  For a general dimension $d$, and for $n_1,\ldots,n_d \in \mathbb{N}$, we define this lattice by $\mathbb{L}_{d,n_1,\ldots,n_d} := \prod_{j=1}^d \{1,\ldots,n_j\}$; when $n_1 = \ldots = n_d = n^{1/d}$ for some $n \in \mathbb{N}$, we also write $\mathbb{L}_{d,n} := \mathbb{L}_{d,n_1,\ldots,n_d}$ as shorthand.  When $\{X_1,\ldots,X_n\} = \mathbb{L}_{2,n_1,n_2}$, \citet{ChatterjeeGuntuboyinaSen2017} showed that there exists a universal constant $C > 0$ such that 
\begin{equation}
  \label{Eq:WorstCase2d}
  R(\hat\theta_n,\theta_0)\leq C\biggl\{\frac{((\theta_0)_{n_1,n_2}-(\theta_0)_{1,1})\log^4 n}{n^{1/2}} + \frac{\log^8 n}{n}\biggr\};
\end{equation}
with a corresponding minimax lower bound of order $n^{-1/2}$.  They also provided a sharp oracle inequality of the form
\begin{equation}
 \label{Eq:Oracle2d}
 R(\hat\theta_n, \theta_0) \leq \inf_{\theta\in\mathcal{M}(\mathbb{L}_{2,n_1,n_2})}\biggl(\frac{\|\theta-\theta_0\|_2^2}{n} + \frac{Ck(\theta)\log^8 n}{n}\biggr),
\end{equation}
where $k(\theta)$ is the minimal number of rectangular blocks into which $\mathbb{L}_{2,n_1,n_2}$ may be partitioned such that $\theta_0$ is constant on each rectangular block. 

A separate line of work has generalised the univariate isotonic regression problem to multivariate settings by assuming an additive structure (see e.g.\ \citet{Bacchetti1989, MortonJonesetal2000, MammenYu2007, ChenSamworth2016}). In the simplest setting, these works investigate the regression problem~\eqref{Eq:Model}, where the signal $f_0$ belongs to 
\[
 \mathcal{F}_d^{\mathrm{add}} := \biggl\{f\in\mathcal{F}_d : f(x_1,\ldots,x_d) = \sum_{j=1}^d f_j(x_j), f_j\in\mathcal{F}_1, \|f_j\|_\infty \leq 1\biggr\}.
\]
The additive structure greatly reduces the complexity of the class; indeed, it can be shown that the least squares estimator over $\mathcal{F}_d^{\mathrm{add}}$ attains the univariate risk $n^{-2/3}$, up to multiplicative constants depending on $d$ \citep[e.g.][Theorem~9.1]{vandeGeer2000}.

The main contribution of this paper is to provide risk bounds for the isotonic least squares estimator when $d\geq 3$, both from a worst-case perspective and an adaptation point of view.  Specifically, we show that in the fixed lattice design case, the least squares estimator satisfies 
\begin{equation}
 \label{Eq:WorstCase3d}
 \sup_{\theta_0 \in \mathcal{M}(\mathbb{L}_{d,n}), \|\theta_0\|_\infty \leq 1} R(\hat{\theta}_n,\theta_0) \leq Cn^{-1/d}\log^4 n,
\end{equation}
for some universal constant $C > 0$. This rate turns out to be the minimax risk up to poly-logarithmic factors in this problem. Furthermore, we establish a sharp oracle inequality: there exists a universal constant $C > 0$ such that for every $\theta_0 \in \mathbb{R}^{\mathbb{L}_{d,n}}$,
\begin{equation}
\label{Eq:Oracle3d}
 R(\hat{\theta}_n,\theta_0) \leq \inf_{\theta\in\mathcal{M}(\mathbb{L}_{d,n})}\biggl\{\frac{\|\theta-\theta_0\|_2^2}{n} + C\biggl(\frac{k(\theta)}{n}\biggr)^{2/d}\log^8 \biggl(\frac{en}{k(\theta)}\biggr)\biggr\},
\end{equation}
where $k(\theta)$ is the number of constant hyperrectangular pieces in $\theta$. This reveals an adaptation rate of nearly $(k/n)^{2/d}$ for signals that are close to an element of $\mathcal{M}(\mathbb{L}_{d,n})$ that has at most $k$ hyperrectangular blocks.  A corresponding lower bound is also provided, showing that the least squares estimator cannot adapt faster than the $n^{-2/d}$ rate implied by~\eqref{Eq:Oracle3d} even for constant signal vectors.  We further demonstrate that the worst-case bounds and oracle inequalities~\eqref{Eq:WorstCase3d} and~\eqref{Eq:Oracle3d}, with slightly different poly-logarithmic exponents, remain valid for random design points $X_1,\ldots,X_n$ sampled independently from a distribution on $[0,1]^d$ with a Lebesgue density bounded away from $0$ and $\infty$. The results in the case of random design are novel even for dimension $d = 2$. These results are surprising in particular with regard to the following two aspects:
\begin{enumerate}
\item The negative results of \citet{BirgeMassart1993} have spawned a heuristic belief that one should not use global empirical risk minimisation procedures\footnote{The term `global' refers here to procedures that involve minimisation over the entire function class, as opposed to only over a sieve; cf. \citet{vandeGeer2000}.} when the entropy integral for the corresponding function class diverges (e.g.\ \citet[p.~121--122]{vandeGeer2000}, \citet{RakhlinSridharanTsybakov2017}).  It is therefore of particular interest to see that in our isotonic regression function setting, the global least squares estimator is still rate optimal (up to poly-logarithmic factors).  See also the discussion after Corollary~\ref{Cor:Fixed}.


 \item Sharp adaptive behaviour for shape-constrained estimators has previously only been shown when the adaptive rate is nearly parametric (see, e.g., \citet{GuntuboyinaSen2015,ChatterjeeGuntuboyinaSen2015,Bellec2017,KGS2017}). On the other hand, our results here show that the least squares estimator in the $d$-dimensional isotonic regression problem necessarily adapts at a strictly nonparametric rate. Clearly, the minimax optimal rate for constant functions is parametric. Hence, the least squares estimator in this problem adapts at a strictly suboptimal rate while at the same time being nearly rate optimal from a worst-case perspective.
\end{enumerate}
In both the fixed lattice design and the more challenging random design cases, our analyses are based on a novel combination of techniques from empirical process theory, convex geometry and combinatorics.  We hope these methods can serve as a useful starting point towards understanding the behaviour of estimators in other multivariate shape-restricted models.

The rest of the paper is organised as follows. In Section~\ref{Sec:FixedDesign}, we state the main results for the fixed lattice design model. Section~\ref{Sec:RandomDesign} describes corresponding results in the random design case. Proofs of all main theoretical results are contained in Sections~\ref{Sec:ProofFixed} and~\ref{Sec:ProofRandom}, whereas proofs of ancillary results are deferred until Section~\ref{Sec:Appendix}.

\subsection{Notation}
For a real-valued measurable function $f$ defined on a probability space $(\mathcal{X},\mathcal{A},P)$ and for $p \in [1,\infty)$, we let $\|f\|_{L_p(P)} := \big(P |f|^p)^{1/p}$ denote the usual $L_p(P)$-norm, and write $\|f\|_\infty := \sup_{x \in \mathcal{X}} |f(x)|$.  For $r\geq 0$, we write $B_p(r, P):=\{f:\mathcal{X}\to \mathbb{R}, \|f\|_{L_p(P)}\leq r\}$ and $B_\infty(r) := \{f: \mathcal{X}\to\mathbb{R}, \|f\|_\infty\leq r\}$. We will abuse notation slightly and also write $B_p(r) := \{v\in\mathbb{R}^n: \|v\|_p \leq r\}$ for $p \in [1,\infty]$. The Euclidean inner product on $\mathbb{R}^d$ is denoted by $\langle \cdot , \cdot \rangle$.  For $x,y\in\mathbb{R}^{d}$, we write $x\preceq y$ if $x_j\leq y_j$ for all $j=1,\ldots,d$.

For $\varepsilon > 0$, the $\varepsilon$-\emph{covering number} of a (semi-)normed space $(\mathcal{F}, \|\cdot\|)$, denoted $N\bigl(\varepsilon, \mathcal{F}, \|\cdot\|\bigr)$, is the smallest number of closed $\varepsilon$-balls whose union covers $\mathcal{F}$. The $\varepsilon$-\emph{bracketing number}, denoted $N_{[\,]}(\varepsilon,\mathcal{F},\|\cdot\|)$, is the smallest number of $\varepsilon$-brackets, of the form $[l, u] := \{f\in\mathcal{F}: l \leq f\leq u\}$ such that $\|u-l\| \leq \varepsilon$, and whose union covers $\mathcal{F}$. The \emph{metric/bracketing entropy} is the logarithm of the covering/bracketing number.

Throughout the article $\epsilon_1,\ldots,\epsilon_n$ and $\{\epsilon_w: w \in \mathbb{L}_{d,n_1,\ldots,n_d}\}$ denote independent standard normal random variables and $\xi_1,\ldots,\xi_n$ denote independent Rademacher random variables, both independent of all other random variables. For two probability measures $P$ and $Q$ defined on the same measurable space $(\mathcal{X}, \mathcal{A})$, we write $d_{\mathrm{TV}}(P,Q) := \sup_{A\in \mathcal{A}} |P(A)- Q(A)|$ for their total variation distance, and $d_{\mathrm{KL}}^2(P,Q) := \int_\mathcal{X} \log\frac{\mathrm{d}P}{\mathrm{d}Q}\,\mathrm{d}P$ for their Kullback--Leibler divergence.

We use $c,C$ to denote generic universal positive constants and use $c_x, C_x$ to denote generic positive constants that depend only on $x$. Exact numeric values of these constants may change from line to line unless otherwise specified. Also, $a\lesssim_{x} b$ and $a\gtrsim_x b$ mean $a\leq C_x b$ and $a\geq c_x b$ respectively, and $a\asymp_x b$ means $a\lesssim_{x} b$ and $a\gtrsim_x b$ ($a\lesssim b$ means $a\leq Cb$ for some absolute constant $C$). We also define $\log_+(x) := \log(x\vee e)$.

\section{Fixed lattice design}
\label{Sec:FixedDesign}
In this section, we focus on the model~\eqref{Eq:Model} in the case where the set of design points forms a finite cubic lattice $\mathbb{L}_{d,n}$, defined in the introduction.  In particular, we will assume in this section that $n = n_1^d$ for some $n_1\in\mathbb{N}$. We use the same notation $\mathbb{L}_{d,n}$ both for the set of points and the directed acyclic graph on these points with edge structure arising from the natural partial ordering induced by $\preceq$. Thus, in the case $d=1$, the graph $\mathbb{L}_{1,n}$ is simply a directed path, and this is the classical univariate isotonic regression setting. The case $d = 2$ is studied in detail in \citet{ChatterjeeGuntuboyinaSen2017}. Our main interest lies in the cases $d\geq 3$.  

\subsection{Minimax rate-optimality of least squares estimator}
Our first result provides an upper bound on the risk of the least squares estimator $\hat\theta_n = \hat\theta_n(Y_1,\ldots,Y_n)$ of $\theta_0 \in \mathcal{M}(\mathbb{L}_{d,n})$.
\begin{thm}
 \label{Thm:FixedUpper}
 Let $d\geq 2$. There exists a universal constant $C > 0$ such that
 \[
  \sup_{\theta_0\in\mathcal{M}(\mathbb{L}_{d,n})\cap B_\infty(1)} R(\hat\theta_n, \theta_0) \leq C n^{-1/d}\log^4 n.
 \]
\end{thm}
Theorem~\ref{Thm:FixedUpper} reveals that, up to a poly-logarithmic factor, the empirical risk of the least squares estimator converges to zero at rate $n^{-1/d}$.  The upper bound in Theorem~\ref{Thm:FixedUpper} is matched, up to poly-logarithmic factors, by the following minimax lower bound. 
\begin{prop}
 \label{Prop:FixedLower}
 There exists a constant $c_d > 0$, depending only on $d$, such that for $d\geq 2$,
 \[
  \inf_{\tilde\theta_n} \sup_{\theta_0\in\mathcal{M}(\mathbb{L}_{d,n}) \cap B_\infty(1)} R(\tilde\theta_n, \theta_0) \geq c_d n^{-1/d},
 \]
 where the infimum is taken over all estimators $\tilde\theta_n = \tilde\theta_n(Y_1,\ldots,Y_n)$ of $\theta_0$.
\end{prop}
From Theorem~\ref{Thm:FixedUpper} and Proposition~\ref{Prop:FixedLower}, together with existing results mentioned in the introduction for the case $d=1$, we see that the worst-case risk $n^{-\min\{2/(d+2),1/d\}}$ (up to poly-logarithmic factors) of the least squares estimator exhibits different rates of convergence in dimension $d=1$ and dimensions $d\geq 3$, with $d=2$ being a transitional case.  From the proof of Proposition~\ref{Prop:FixedLower}, we see that it is the competition between the cardinality of the maximum chain (totally ordered subset) and the maximum antichain (subset of mutually incomparable design points) that explains the different rates.  Similar transitional behaviour was recently observed by \citet{KimSamworth2016} in the context of log-concave density estimation, though there it is the tension between estimating the density in the interior of its support and estimating the support itself that drives the transition.

The two results above can readily be translated into bounds for the rate of convergence for estimation of a block monotonic function with a fixed lattice design.  Recall that $\mathcal{F}_d$ is the class of block increasing functions.  Suppose that for some $f_0 \in \mathcal{F}_d$, and at each $x = (x_1,\ldots,x_d) \in n_1^{-1}\mathbb{L}_{d,n}$, where $n_1 = n^{1/d}$, we observe $Y(x) \sim N(f_0(x), 1)$ independently. Define $P_n := n^{-1}\sum_{x\in n_1^{-1}\mathbb{L}_{d,n}} \delta_x$ and let $\mathcal{A}$ denote the set of hypercubes of the form $A = \prod_{j=1}^d A_j$, where either $A_j = [0, \frac{1}{n_1}]$ or $A_j = (\frac{i_j-1}{n_1}, \frac{i_j}{n_1}]$ for some $i_j \in \{2,\ldots,n_1\}$.  Now let $\mathcal{H}$ denote the set of functions $f \in \mathcal{F}_d$ that are piecewise constant on each $A \in \mathcal{A}$, and set
\[
\hat f_n:= \argmin_{f \in \mathcal{H}} \frac{1}{n}\sum_{i=1}^n \{Y(x_i) - f(x_i)\}^2.
\]
The following is a fairly straightforward corollary of Theorem~\ref{Thm:FixedUpper} and Proposition~\ref{Prop:FixedLower}.
\begin{cor}
\label{Cor:Fixed}
There exist constants $c_d, C_d > 0$, depending only on $d$, such that for $Q = P_n$ or Lebesgue measure on $[0,1]^d$, we have
 \[
  c_d n^{-1/d} \leq \inf_{\tilde f_n} \sup_{f_0\in\mathcal{F}_d \cap B_\infty(1)} \mathbb{E} \|\tilde f_n - f_0\|^2_{L_2(Q)} \leq \sup_{f_0\in\mathcal{F}_d \cap B_\infty(1)} \mathbb{E} \|\hat f_n- f_0\|^2_{L_2(Q)} \leq C_d n^{-1/d}\log^4 n,
 \]
 where the infimum is taken over all measurable functions of $\{Y(x):x\in n_1^{-1}\mathbb{L}_{d,n}\}$.
\end{cor}
This corollary is surprising in the following sense.  \citet[][Theorem~1.1]{GaoWellner2007} proved that for $d\geq 3$, 
\begin{equation}
\label{Eq:GaoWellner}
 \log N\bigl(\varepsilon, \mathcal{F}_d \cap B_\infty(1), \|\cdot\|_2\bigr) \asymp_d \varepsilon^{-2(d-1)}.
\end{equation}
In particular, for $d\geq 3$, the classes $\mathcal{F}_d \cap B_\infty(1)$ are massive in the sense that the entropy integral $\int_{\delta}^1 \log^{1/2} N(\varepsilon, \mathcal{F}_d \cap B_\infty(1), \|\cdot\|_2) \, \mathrm{d}\varepsilon$ diverges at a polynomial rate in $\delta^{-1}$ as $\delta \searrow 0$.  To the best of our knowledge, this is the first example of a setting where a global empirical risk minimisation procedure has been proved to attain (nearly) the minimax rate of convergence over such massive parameter spaces.

\subsection{Sharp oracle inequality}

In this subsection, we consider the adaptation behaviour of the least squares estimator in dimensions $d \geq 2$ (again, the $d=2$ case is covered in \citet{ChatterjeeGuntuboyinaSen2017}).  Our main result is the sharp oracle inequality in Theorem~\ref{Thm:FixedOracle} below. We call a set in $\mathbb{R}^d$ a hyperrectangle if it is of the form $\prod_{j=1}^d I_j$ where $I_j\subseteq \mathbb{R}$ is an interval for each $j = 1,\ldots,d$.  By a slight abuse of terminology, we also call a subset of $\mathbb{L}_{d,n}$ a hyperrectangle if it is the intersection of a hyperrectangle in $[0,1]^d$ and $\mathbb{L}_{d,n}$.  We say a subset $A$ of $\mathbb{L}_{d,n}$ is a \emph{two-dimensional sheet} if $A = \prod_{j=1}^d [a_j, b_j]$ where $|\{j: b_j = a_j\}|\geq d-2$.  A two-dimensional sheet is therefore a special type of hyperrectangle whose intrinsic dimension is at most two.  For $\theta\in\mathcal{M}(\mathbb{L}_{d,n})$, let $K(\theta)$ denote the cardinality of the minimal partition $\mathbb{L}_{d,n} = \sqcup_{\ell=1}^K A_\ell$ of $\mathbb{L}_{d,n}$ into a disjoint union of two-dimensional sheets $A_1,\ldots,A_K$, where the restricted vector $\theta_{A_\ell} = (\theta(u))_{u \in A_\ell}$ is constant for each $\ell=1,\ldots,K$.
\begin{thm}
  \label{Thm:FixedOracle}
  Let $d\geq 2$. There exists a universal constant $C>0$ such that for every $\theta_0\in\mathbb{R}^{\mathbb{L}_{d,n}}$,
  \[
    R(\hat\theta_n, \theta_0) \leq \inf_{\theta\in\mathcal{M}(\mathbb{L}_{d,n})}\biggl\{\frac{\|\theta-\theta_0\|_2^2}{n} + \frac{CK(\theta)}{n}\log_+^8 \biggl(\frac{n}{K(\theta)}\biggr)\biggr\}.
  \]
\end{thm}

We remark that Theorem~\ref{Thm:FixedOracle} does not imply (nearly) parametric adaptation when $d \geq 3$. This is because even when $\theta_0$ is constant on $\mathbb{L}_{d,n}$ for every $n$, we have $K(\theta_0) = n^{(d-2)/d} \rightarrow \infty$ as $n\to\infty$. The following corollary of Theorem~\ref{Thm:FixedOracle} gives an alternative (weaker) form of oracle inequality that offers easier comparison to lower dimensional results given in~\eqref{Eq:Oracle1d} and~\eqref{Eq:Oracle2d}. Let $\mathcal{M}^{(k)}(\mathbb{L}_{d,n})$ be the collection of all $\theta\in\mathcal{M}(\mathbb{L}_{d,n})$  such that there exists a partition $\mathbb{L}_{d,n} = \sqcup_{\ell=1}^k R_\ell$ where $R_1,\ldots,R_k$ are hyperrectangles with the property that for each $\ell$, the restricted vector $\theta_{R_\ell}$ is constant.
\begin{thm}
 \label{Thm:FixedOracleBlock}
 Let $d\geq 2$. There exists a universal constant $C>0$ such that for every $\theta_0 \in \mathbb{R}^{\mathbb{L}_{d,n}}$,
 \[
  R(\hat\theta_n, \theta_0) \leq \inf_{k\in\mathbb{N}}\biggl\{\inf_{\theta\in\mathcal{M}^{(k)}(\mathbb{L}_{d,n})} \frac{\|\theta-\theta_0\|_2^2}{n} + C\biggl(\frac{k}{n}\biggr)^{2/d} \log_+^8 \biggl(\frac{n}{k}\biggr)\biggr\}.
 \]
\end{thm}
It is important to note that both Theorems~\ref{Thm:FixedOracle} and~\ref{Thm:FixedOracleBlock} allow for model misspecification, as it is not assumed that $\theta_0 \in \mathcal{M}(\mathbb{L}_{d,n})$. For signal vectors $\theta_0$ that are piecewise constant on $k$ hyperrectangles, Theorem~\ref{Thm:FixedOracleBlock} provides an upper bound of the risk of order $(k/n)^{2/d}$ up to poly-logarithmic factors.  The following proposition shows that even for a constant signal vector, the adaptation rate of $n^{-2/d}$ given in Theorem~\ref{Thm:FixedOracleBlock} cannot be improved. 
\begin{prop}
 \label{Prop:FixedAdaptLower}
 Let $d\geq 2$.  There exists a constant $c_d > 0$, depending only on $d$, such that for any $\theta_0\in\mathcal{M}^{(1)}(\mathbb{L}_{d,n})$, 
 \[
   R(\hat\theta_n, \theta_0) \geq c_d\begin{cases} n^{-1}\log^2 n & \text{if $d = 2$}\\ n^{-2/d} & \text{if $d\geq 3$.}\end{cases}
 \]
\end{prop}
The case $d=2$ of this result is new, and reveals both a difference with the univariate situation, where the adaptation rate is of order $n^{-1}\log n$ \citep{Bellec2017}, and that a poly-logarithmic penalty relative to the parametric rate is unavoidable for the least squares estimator.  Moreover, we see from Proposition~\ref{Prop:FixedAdaptLower} that for $d \geq 3$, although the least squares estimator achieves a faster rate of convergence than the worst-case bound in Theorem~\ref{Thm:FixedUpper} on constant signal vectors, the rate is not parametric, as would have been the case for a minimax optimal estimator over the set of constant vectors. This is in stark contrast to the nearly parametric adaptation results established in~\eqref{Eq:Oracle1d} and~\eqref{Eq:Oracle2d} for dimensions $d\leq 2$. 

Another interesting aspect of these results relates to the notion of \emph{statistical dimension}, defined for an arbitrary cone $C$ in $\mathbb{R}^n$ by\footnote{Our reason for defining the statistical dimension via an integral rather than as $\mathbb{E}\|\Pi_C(\epsilon)\|_2^2$ is because, in the random design setting, the cone $C$ is itself random, and in that case $\delta(C)$ is a random quantity.} $\delta(C) := \int_{\mathbb{R}^n} \|\Pi_C(x)\|_2^2 (2\pi)^{-n/2}e^{-\|x\|_2^2/2} \, \mathrm{d}x$, where $\Pi_{C}$ is the projection onto the set $C$ \citep{Amelunxenetal2014}.  Theorem~\ref{Thm:FixedOracleBlock} and Proposition~\ref{Prop:FixedAdaptLower} reveal a type of phase transition phenomenon for the statistical dimension $\delta(\mathcal{M}(\mathbb{L}_{d,n})) = R(\hat\theta_n, 0)$ of the monotone cone (cf.\ Table~\ref{Tab:Summary}).
\begin{table}[htbp!]
    \begin{center}
    \begin{threeparttable}
            \caption{Bounds\tnote{$\ast$} for $\delta\bigl(\mathcal{M}(\mathbb{L}_{d,n})\bigr)$.}    \label{Tab:Summary}
            \begin{tabular}{ccccc}
                $d$ & & upper bound & & lower bound\\
                \hline\\[-12pt]
                $1$ & &  $\sum_{i=1}^n i^{-1}$ \tnote{$\dagger$} & & $\sum_{i=1}^n i^{-1}$ \tnote{$\dagger$} \vspace{0.2cm}\\
                $2$ & &  $\lesssim \log^8 n$  \tnote{$\ddagger$}& & $\gtrsim \log^2 n$ \vspace{0.2cm}\\
                $\geq 3$ & & $\lesssim n^{1-2/d}\log^8 n$ & & $\gtrsim_d n^{1-2/d}$
            \end{tabular}
            {\footnotesize \begin{tablenotes}
                \item[$\ast$] Entries without a reference are proved in this paper.
                \item[$\dagger$] \citet{Amelunxenetal2014}
                \item[$\ddagger$] \citet{ChatterjeeGuntuboyinaSen2017}
            \end{tablenotes}}
    \end{threeparttable}
    \end{center}
\end{table}

The following corollary of Theorem~\ref{Thm:FixedOracle} gives another example where different adaptation behaviour is observed in dimensions $d\geq 3$, in the sense that the $n^{-2/d}\log^8 n$ adaptive rate achieved for constant signal vectors is actually available for a much wider class of isotonic signals that depend only on $d-2$ of all $d$ coordinates of $\mathbb{L}_{d,n}$. For $r = 0,1,\ldots,d$, we say a vector $\theta_0\in \mathcal{M}(\mathbb{L}_{d,n})$ is a \emph{function of $r$ variables}, written $\theta_0\in\mathcal{M}_r(\mathbb{L}_{d,n})$, if there exists $\mathcal{J} \subseteq \{1,\ldots,d\}$, of cardinality~$r$, such that $(\theta_0)_{(x_1,\ldots,x_d)} = (\theta_0)_{(x_1',\ldots,x_d')}$ whenever $x_j = x_j'$ for all $j \in \mathcal{J}$.   
\begin{cor}
  \label{Cor:d-2}
  For $d \geq 2$, there exists constant $C_d>0$, depending only on $d$, such that 
  \[
    \sup_{\theta_0\in\mathcal{M}_r(\mathbb{L}_{d,n}) \cap B_\infty(1)} R(\hat\theta_n, \theta_0) \leq C_d\begin{cases}
                                                                  n^{-2/d}\log^8 n & \text{ if $r \leq d-2$}\\
                                                                  n^{-4/(3d)}\log^{16/3} n & \text{ if $r = d-1$}\\
                                                                  n^{-1/d}\log^4 n & \text{ if $r = d$.}
                                                                 \end{cases}
  \]
\end{cor}
If the signal vector $\theta_0$ belongs to $\mathcal{M}_r(\mathbb{L}_{d,n})$, then it is intrinsically an $r$-dimensional isotonic signal. Corollary~\ref{Cor:d-2} demonstrates that the least squares estimator exhibits three different levels of adaptation when the signal is a function of $d, d-1, d-2$ variables respectively. However, viewed together with Proposition~\ref{Prop:FixedLower}, Corollary~\ref{Cor:d-2} shows that no further adaptation is available when the intrinsic dimension of the signal vector decreases further. Moreover, if we let $\tilde n = n^{2/d}$ denote the size of a maximal two-dimensional sheet in $\mathbb{L}_{d,n}$, then the three levels of adaptive rates in Corollary~\ref{Cor:d-2} are $\tilde n^{-1}$, $\tilde n^{-2/3}$ and $\tilde n^{-1/2}$ respectively, up to poly-logarithmic factors, matching the two-dimensional `automatic variable adaptation' result described in \citet[Theorem~2.4]{ChatterjeeGuntuboyinaSen2017}. In this sense, the adaptation of the isotonic least squares estimator in general dimensions is essentially a two-dimensional phenomenon.

\section{Random design}
\label{Sec:RandomDesign}
In this section, we consider the setting where the design points $X_1,\ldots,X_n$ are independent and identically distributed from some distribution $P$ supported on the unit cube $[0,1]^d$.  We will assume throughout that $P$ has Lebesgue density $p_0$ such that $0<m_0\leq \inf_{x\in[0,1]^d}p_0(x) \leq \sup_{x\in[0,1]^d}p_0(x)\leq M_0 <\infty$.  Since the least squares estimator $\hat{f}_n$ is only well-defined on $X_1,\ldots,X_n$, for definiteness, we extend $\hat{f}_n$ to $[0,1]^d$ by defining $\hat{f}_n(x) := \min\bigl(\{\hat f_n(X_i): 1\leq i\leq n, X_i\succeq x\} \cup \{ \max_i \hat f_n(X_i) \} \bigr)$.  If we let $\mathbb{P}_n := n^{-1}\sum_{i=1}^n \delta_{X_i}$, then the risk function~\eqref{Eq:Loss} is $R(\hat{f}_n, f_0) = \mathbb{E} \|\hat{f}_n - f_0\|_{L_2(\mathbb{P}_n)}^2$ in the context of random design.

The main results of this section are the following two theorems, establishing respectively the worst-case performance and the sharp oracle inequality for the least squares estimator in the random design setting. We write $\mathcal{F}_d^{(k)}$ for the class of functions in $\mathcal{F}_d$ that are piecewise constant on $k$ hyperrectangular pieces. In other words, if $f\in\mathcal{F}_d^{(k)}$, then there exists a partition $[0,1]^d = \sqcup_{\ell=1}^k R_\ell$, such that the closure of each $R_\ell$ is a hyperrectangle and $f$ is a constant function when restricted to each $R_\ell$.  Let $\gamma_2 := 9/2$ and $\gamma_d := (d^2+d+1)/2$ for $d \geq 3$.
\begin{thm}
  \label{Thm:RandomUpper}
  Let $d\geq 2$. There exists a constant $C_{d,m_0,M_0}>0$, depending only on $d, m_0$ and $M_0$, such that
  \[
    \sup_{f_0\in\mathcal{F}_d\cap B_\infty(1)} R(\hat{f}_n,f_0) \leq C_{d,m_0,M_0} n^{-1/d}\log^{\gamma_d} n.
  \]
\end{thm}

\begin{thm}
  \label{Thm:RandomOracle}
  Let $d\geq 2$.  There exists a constant $C_{d,m_0,M_0}>0$,  depending only on $d, m_0$ and $M_0$, such that for any measurable function $f_0: [0,1]^d \to \mathbb{R}$, we have
  \[
    R(\hat{f}_n,f_0) \leq \inf_{k\in\mathbb{N}} \biggl\{ \inf_{f\in\mathcal{F}_d^{(k)}} \|f-f_0\|_{L_2(P)}^2 + C_{d,m_0,M_0} \biggl(\frac{k}{n}\biggr)^{2/d}\log_+^{2\gamma_d} \biggl(\frac{n}{k}\biggr)\biggr\}.
  \]
\end{thm}

To the best of our knowledge, Theorem~\ref{Thm:RandomOracle} is the first sharp oracle inequality in the shape-constrained regression literature with random design.  The different norms on the left- and right-hand sides arise from the simple observation that $\mathbb{E}\|f-f_0\|_{L_2(\mathbb{P}_n)}^2 = \|f-f_0\|_{L_2(P)}^2$ for $f\in\mathcal{F}_d^{(k)}$.  The proofs of Theorems~\ref{Thm:RandomUpper} and~\ref{Thm:RandomOracle} are considerably more involved than those of the corresponding Theorems~\ref{Thm:FixedUpper} and~\ref{Thm:FixedOracle} in Section~\ref{Sec:FixedDesign}.  We briefly mention two major technical difficulties:
\begin{enumerate}
\item The size of $\mathcal{F}_d$, as measured by its entropy, is large when $d \geq 3$, even after $L_\infty$ truncation (cf.~\eqref{Eq:GaoWellner}).  As rates obtained from the entropy integral \citep[e.g.][Theorem~9.1]{vandeGeer2000} do not match those from Sudakov lower bounds for such classes, standard entropy methods result in a non-trivial gap between the minimax rates of convergence, which typically match the Sudakov lower bounds \citep[e.g.][Proposition~1]{YangBarron1999}, and provable risk upper bounds for least squares estimators when $d\geq 3$. 

\item In the fixed lattice design case, our analysis circumvents the difficulties of standard entropy methods by using the fact that a $d$-dimensional cubic lattice can be decomposed into a union of lower-dimensional pieces. This crucial property is no longer valid when the design is random.
\end{enumerate}  

We do not claim any optimality of the power in the poly-logarithmic factor in the oracle inequality in Theorems~\ref{Thm:RandomUpper} and~\ref{Thm:RandomOracle}.  On the other hand, similar to the fixed, lattice design case, the worst-case rate $n^{-1/d}$ and adaptation rate $n^{-2/d}$ cannot be improved, as can be seen from the following two propositions. 

\begin{prop}
  \label{Prop:RandomLower}
  Let $d \geq 2$.  There exists a constant $c_{d,m_0,M_0} > 0$, depending only on $d, m_0$ and $M_0$, such that,
  \[
    \inf_{\tilde f_n}\sup_{f_0\in\mathcal{F}_d \cap B_\infty(1)} R(\tilde{f}_n,f_0) \geq c_{d,m_0,M_0} n^{-1/d},
  \]
  where the infimum is taken over all measurable functions $\tilde f_n$ of the data $(X_1, Y_1),\ldots,(X_n,Y_n)$.
\end{prop}

\begin{prop}
  \label{Prop:RandomAdaptLower}
  Let $d \geq 2$.  There exists a constant $c_{d,M_0} > 0$, depending only on $d$ and $M_0$, such that for any $f_0\in\mathcal{F}^{(1)}_d$,
  \[
    R(\hat{f}_n,f_0) \geq c_{d,M_0} n^{-2/d}.
  \]
\end{prop}
A key step in proving Proposition~\ref{Prop:RandomAdaptLower} is to establish that with high probability, the cardinality of the maximum antichain in $G_X$ is at least of order $n^{1-1/d}$.  When $d=2$, the distribution of this maximum cardinality is the same as the distribution of the length of the longest decreasing subsequence of a uniform permutation of $\{1,\ldots,n\}$, a famous object of study in probability and combinatorics.  See \citet{Romik2014} and references therein.

\section{Proofs of results in Section~\ref{Sec:FixedDesign}}
\label{Sec:ProofFixed}
Throughout this section, $\epsilon = (\epsilon_w)_{w \in \mathbb{L}_{d,n_1,\ldots,n_d}}$ denotes a vector of independent standard normal random variables.  It is now well understood that the risk of the least squares estimator in the Gaussian sequence model is completely characterised by the size of a localised Gaussian process; cf.\ \citet{Chatterjee2014}.  The additional cone property of $\mathcal{M}(\mathbb{L}_{d,n})$ makes the reduction even simpler: we only need to evaluate the Gaussian complexity of $\mathcal{M}(\mathbb{L}_{d,n})\cap B_2(1)$, where the \emph{Gaussian complexity} of $T\subseteq \mathbb{R}^{\mathbb{L}_{d,n_1,\ldots,n_d}}$ is defined as $w_T := \mathbb{E} \sup_{\theta\in T}\langle \epsilon, \theta\rangle$.  Thus the result in the following proposition constitutes a key ingredient in analysing the risk of the least squares estimator. 

\begin{prop}
\label{Prop:GaussianWidthLattice}
There exists a universal constant $C>0$ such that for $d\geq 2$ and every $1\leq n_1\leq \cdots \leq n_d$ with $\prod_{j=1}^d n_j = n$, we have
\[
\frac{\sqrt{2/\pi}}{(d-1)^{d-1}} n_1^{d-1}n^{-1/2} \leq \mathbb{E} \sup_{\theta \in \mathcal{M}(\mathbb{L}_{d,n_1,\ldots,n_d})\cap B_2(1)} \langle \epsilon , \theta\rangle \leq C\sqrt{\frac{n}{n_{d-1}n_d}}\log^4 n.
\]
\end{prop}
\begin{remark}
 In the case $n_1=\cdots=n_d=n^{1/d}$, we have $n_1^{d-1}n^{-1/2} = \sqrt{\frac{n}{n_{d-1}n_d}} = n^{1/2 - 1/d}$.
\end{remark}
\begin{remark}
From the symmetry of the problem, we see that the restriction that $n_1 \leq \cdots \leq n_d$ is not essential.  In the general case, for the lower bound, $n_1$ should be replaced with $\min_j n_j$, while in the upper bound, $n_{d-1}n_d$ should be replaced with the product of the two largest elements of $\{n_1,\ldots,n_d\}$ (considered here as a multiset).
\end{remark}

\begin{proof}
We first prove the lower bound.  Consider the set $W := \{w \in \mathbb{L}_{d,n_1,\ldots,n_d}: \sum_{j=1}^d w_j= n_1\}$, and define $W^+ := \{w \in \mathbb{L}_{d,n_1,\ldots,n_d}: \sum_{j=1}^d w_j> n_1\}$ and $W^- := \{w \in \mathbb{L}_{d,n_1,\ldots,n_d}: \sum_{j=1}^d w_j< n_1\}$. For each realisation of the Gaussian random vector $\epsilon = (\epsilon_w)_{w \in \mathbb{L}_{d,n_1,\ldots,n_d}}$, we define $\theta(\epsilon) = (\theta_w(\epsilon))_{w\in \mathbb{L}_{d,n_1,\ldots,n_d}} \in \mathcal{M}(\mathbb{L}_{d,n_1,\ldots,n_d})$ by
\[
\theta_w := \begin{cases}
1& \text{if $w \in W^+$}\\
\mathrm{sgn}(\epsilon_w)& \text{if $w \in W$}\\
-1 &  \text{if $w \in W^-$}.\\
\end{cases}
\]
Since $\|\theta(\epsilon)\|_2^2=n$, it follows that 
\begin{align*}
\mathbb{E} \sup_{\theta \in \mathcal{M}(\mathbb{L}_{d,n_1,\ldots,n_d}) \cap B_2(1)} \langle \epsilon, \theta\rangle & \geq \mathbb{E} \biggl\langle \epsilon, \frac{\theta(\epsilon)}{\|\theta(\epsilon)\|_2} \biggr\rangle  = \frac{1}{n^{1/2}} \mathbb{E}\biggl( \sum_{w \in W^+}  \epsilon_w -\sum_{w \in W^-} \epsilon_w+ \sum_{w \in W}  |\epsilon_w|\biggr)\\
&= \frac{\sqrt{2/\pi}}{n^{1/2}}|W|.
\end{align*}
The proof of the lower bound is now completed by noting that 
\begin{equation}
\label{Eq:DiagonalCardinality}
|W|=\binom{d+n_1-1}{d-1}\geq \biggl(\frac{n_1}{d-1}\biggr)^{d-1}.
\end{equation}

We next prove the upper bound.  For $j=1,\ldots,d-2$ and $x_j \in \{1,\ldots, n_j\}$, we define $A_{x_1,\dots,x_{d - 2}} := \{w = (w_1,\ldots,w_d)^\top \in \mathbb{L}_{d,n_1,\ldots,n_d}: (w_1,\dots,w_{d - 2}) = (x_1,\dots,x_{d - 2})\}$. Each $A_{x_1,\dots,x_{d - 2}}$ can be viewed as a directed acyclic graph with graph structure inherited from $\mathbb{L}_{d,n_1,\ldots,n_d}$. Since monotonicity is preserved under the subgraph restriction, we have that $\mathcal{M}(\mathbb{L}_{d,n_1,\ldots,n_d})\subseteq \bigoplus_{x_1,\ldots,x_{d-2}}\mathcal{M}(A_{x_1,\ldots,x_{d-2}})$. Therefore, by the Cauchy--Schwarz inequality, \citet[Proposition~3.1(5, 9, 10)]{Amelunxenetal2014} and \citet[Theorem~2.1]{ChatterjeeGuntuboyinaSen2017}, we obtain that
\begin{align*}
\biggl(\mathbb{E} \sup_{\theta \in \mathcal{M}(\mathbb{L}_{d,n_1,\ldots,n_d})\cap B_2(1)} \langle \epsilon , \theta\rangle \biggr)^2 &\leq \mathbb{E} \biggl\{\biggl(\sup_{\theta \in \mathcal{M}(\mathbb{L}_{d,n_1,\ldots,n_d})\cap B_2(1)} \langle \epsilon , \theta\rangle \biggr)^2\biggr\} \\
&= \delta\bigl(\mathcal{M}(\mathbb{L}_{d,n_1,\ldots,n_d})\bigr) \leq \sum_{x_1,\ldots,x_{d-2}} \delta \bigl( \mathcal{M}(A_{x_1,\ldots,x_{d-2}}) \bigr) \\
&=   \delta\bigl(\mathcal{M}(\mathbb{L}_{2,n_{d-1},n_d})\bigr) \prod_{j=1}^{d-2} n_j \lesssim \frac{n}{n_{d-1}n_d}\log^8 (en_{d-1}n_d),
\end{align*}
as desired.
\end{proof}
\begin{proof}[Proof of Theorem~\ref{Thm:FixedUpper}]
Fix $\theta_0\in\mathcal{M}(\mathbb{L}_{d,n})\cap B_\infty(1)$. By \citet[Theorem~1.1]{Chatterjee2014}, the function
 \[
t\mapsto \mathbb{E}\sup_{\theta\in\mathcal{M}(\mathbb{L}_{d,n}), \|\theta-\theta_0\|\leq t} \langle \epsilon, \theta-\theta_0\rangle - t^2/2  
 \]
 is strictly concave on $[0,\infty)$ with a unique maximum at, say, $t_0 \geq 0$. We note that $t_0\leq t_*$ for any $t_*$ satisfying
 \begin{equation}
  \label{Eq:tstar}
  \mathbb{E} \sup_{\theta\in\mathcal{M}(\mathbb{L}_{d,n}), \|\theta-\theta_0\|\leq t_*} \langle \epsilon, \theta-\theta_0\rangle \leq \frac{t_*^2}{2}.
 \end{equation}
 For a vector $\theta = (\theta_x)_{x \in \mathbb{L}_{d,n}}$, we define $\bar\theta:= n^{-1}\sum_{x\in\mathbb{L}_{d,n}} \theta_x$ and write $\mathbf{1}_n\in \mathbb{R}^{\mathbb{L}_{d,n}}$ for the all-one vector. Then 
 \begin{align*}
  \mathbb{E} \sup_{\theta\in\mathcal{M}(\mathbb{L}_{d,n}), \|\theta-\theta_0\|_2\leq t_*} \langle \epsilon, \theta-\theta_0\rangle & = \mathbb{E} \sup_{\theta\in\mathcal{M}(\mathbb{L}_{d,n}), \|\theta-\theta_0\|_2\leq t_*} \Bigl\{\langle \epsilon, \theta-\bar\theta_0\mathbf{1}_n\rangle + \langle \epsilon, \bar\theta_0\mathbf{1}_n - \theta_0\rangle\Bigr\}\\
  & \leq \mathbb{E} \sup_{\theta\in\mathcal{M}(\mathbb{L}_{d,n}), \|\theta-\bar\theta_0\mathbf{1}_n\|_2\leq t_*+ n^{1/2}} \langle \epsilon, \theta-\bar\theta_0\mathbf{1}_n\rangle\\
  & = \mathbb{E} \sup_{\theta\in\mathcal{M}(\mathbb{L}_{d,n})\cap B_2(t_*+n^{1/2})} \langle \epsilon, \theta\rangle = \bigl\{t_* + n^{1/2}\bigr\} w_{\mathcal{M}(\mathbb{L}_{d,n}) \cap B_2(1)},
 \end{align*}
 where we recall that $w_{\mathcal{M}(\mathbb{L}_{d,n}) \cap B_2(1)} = \mathbb{E}\sup_{\theta\in\mathcal{M}(\mathbb{L}_{d,n})\cap B_2(1)} \langle \epsilon,\theta\rangle$. Therefore, to satisfy~\eqref{Eq:tstar}, it suffices to choose
 \begin{align}
\label{Eq:t*}
  t_* &= w_{\mathcal{M}(\mathbb{L}_{d,n})\cap B_2(1)} + \bigl\{w_{\mathcal{M}(\mathbb{L}_{d,n})\cap B_2(1)}^2 + 2n^{1/2}w_{\mathcal{M}(\mathbb{L}_{d,n})\cap B_2(1)}\bigr\}^{1/2} \nonumber \\
&\lesssim \max\bigl\{w_{\mathcal{M}(\mathbb{L}_{d,n})\cap B_2(1)}, n^{1/4}w_{\mathcal{M}(\mathbb{L}_{d,n})\cap B_2(1)}^{1/2}\bigr\}.
 \end{align}
 Consequently, by \citet[Corollary~1.2]{Chatterjee2014} and Proposition~\ref{Prop:GaussianWidthLattice}, we have that 
 \begin{align*}
  R(\hat\theta_n, \theta_0)  \lesssim n^{-1}\max(1, t_0^2) \lesssim n^{-1}t_*^2 \lesssim n^{-1/d}\log^4 n,
 \end{align*}
 which completes the proof.
\end{proof}

The following proposition is the main ingredient of the proof of the minimax lower bound in Proposition~\ref{Prop:FixedLower}.  It exhibits a combinatorial obstacle, namely the existence of a large antichain, that prevents any estimator from achieving a faster rate of convergence. We state the result in the more general and natural setting of least squares isotonic regression on directed acyclic graphs. Recall that the isotonic regression problem on a directed acyclic graph $G = (V(G),E(G))$ is of the form $Y_v = \theta_v + \epsilon_v$, where $\theta = (\theta_v)_{v\in V(G)} \in \mathcal{M}(G)$ and $\epsilon = (\epsilon_v)_{v\in V(G)}$ is a vector of independent $N(0,1)$ random variables. 

\begin{prop}
\label{Prop:AntichainLowerBound}
If $G = (V(G),E(G))$ is a directed acyclic graph and $W \subseteq V(G)$ is a maximum antichain of $G$, then
 \[
  \inf_{\tilde{\theta}_n} \sup_{\theta_0\in\mathcal{M}(G)\cap B_\infty(1)} R(\tilde\theta_n, \theta_0) \geq \frac{8|W|}{27n},
 \]
where the infimum is taken over all measurable functions $\tilde\theta_n$ of $\{Y_v:v\in V(G)\}$.
\end{prop}
\begin{proof}
If $v\notin W$, then by the maximality of $W$, there exists $u_0\in W$ such that either $u_0\leq v$ or $u_0\geq v$. Suppose without loss of generality it is the former. Then $v\not\leq u$ for any $u\in W$, because otherwise we would have $u_0\leq u$, contradicting the fact that $W$ is an antichain. It follows that we can write $V(G) = W^+ \sqcup W \sqcup W^-$, where for all $v\in W^+$, $u\in W$, we have $u\not\geq v$, and similarly for all $v\in W^-$, $u\in W$, we have $v\not\geq u$.
 
 For $\tau = (\tau_w) \in \{0,1\}^W =: T$, we define $\theta^\tau = (\theta^\tau_v) \in \mathcal{M}(G) \cap B_\infty(1)$ by
 \[
   \theta^\tau_v = \begin{cases}
                     -1 & \text{ if $v\in W^-$}\\
                     \rho(2\tau_v-1) & \text{ if $v\in W$}\\
                     1 & \text{ if $v\in W^+$},
                   \end{cases}
 \]
 where $\rho\in(0,1)$ is a constant to be chosen later. Let $P_\tau$ denote the distribution of $\{Y_v:v \in V(G)\}$ when the isotonic signal is $\theta^\tau$. Then, for $\tau, \tau' \in T$, by Pinsker's inequality \citep[e.g.][p.~62]{Pollard2002}, we have  
 \[
   d_{\mathrm{TV}}^2(P_\tau, P_{\tau'}) \leq \frac{1}{2} d_{\mathrm{KL}}^2(P_\tau,P_{\tau'}) = \frac{n}{4} \|\theta^\tau - \theta^{\tau'}\|_2^2 = n\rho^2\|\tau-\tau'\|_0.
 \]
Consequently, setting $\rho = 2/(3n^{1/2})$, by Assouad's Lemma \citep[cf.][Lemma~2]{Yu1997}, we have that
 \[
   \inf_{\tilde\theta_n}\sup_{\theta_0\in\mathcal{M}(G)\cap B_\infty(1)} R(\tilde\theta_n, \theta_0) \geq \inf_{\tilde\theta_n}\sup_{\tau\in T} \frac{1}{n}\mathbb{E}\|\tilde\theta_n - \theta^\tau\|_2^2 \geq 2\rho^2|W|( 1 - n^{1/2}\rho) = \frac{8|W|}{27n},
 \]
 as desired. 
\end{proof}

\begin{proof}[Proof of Proposition~\ref{Prop:FixedLower}]
  Recall that $n_1 = n^{1/d}$. We note that the set
  \[
   W := \biggl\{v = (v_1,\ldots,v_d)^\top \in \mathbb{L}_{d,n} : \sum_{j=1}^d v_j = n_1\biggr\}
  \]
  is an antichain in $\mathbb{L}_{d,n}$ of cardinality $\binom{d+n_1-1}{d-1}\geq \frac{n^{1-1/d}}{(d-1)^{d-1}}$. Hence any maximum antichain of $\mathbb{L}_{d,n}$ is at least of this cardinality. The desired result therefore follows from Proposition~\ref{Prop:AntichainLowerBound}.
\end{proof}

\begin{proof}[Proof of Corollary~\ref{Cor:Fixed}]
  For $Q = P_n$, the result is an immediate consequence of Theorem~\ref{Thm:FixedUpper} and Proposition~\ref{Prop:FixedLower}, together with the facts that
 \[
   \inf_{\tilde\theta_n}\sup_{\theta_0\in\mathcal{M}(\mathbb{L}_{d,n})\cap B_\infty(1)} R(\tilde\theta_n,\theta_0) = \inf_{\tilde f_n}\sup_{f_0\in\mathcal{F}_d\cap B_\infty(1)} \mathbb{E}\|\tilde f_n - f_0\|_{L_2(P_n)}^2 
 \]
 and
 \[
  \sup_{\theta_0\in\mathcal{M}(\mathbb{L}_{d,n})\cap B_\infty(1)} R(\hat\theta_n,\theta_0) =  \sup_{f_0\in\mathcal{F}_d\cap B_\infty(1)} \mathbb{E} \|\hat f_n - f_0\|_{L_2(P_n)}^2.
 \]
Now suppose that $Q$ is Lebesgue measure on $[0,1]^d$.  For any $f: [0,1]^d \to \mathbb{R}$, we may define $\theta(f) : \mathbb{L}_{d,n}\to \mathbb{R}$ by $\theta(f)(x) := f(n_1^{-1}x)$. On the other hand, for any $\theta : \mathbb{L}_{d,n}\to \mathbb{R}$, we can also define $f(\theta): [0,1]^d \to \mathbb{R}$ by 
 \[
  f(\theta)(x_1,\ldots,x_d) := \theta(\lfloor n_1 x_1\rfloor, \ldots, \lfloor n_1 x_d\rfloor).
 \]
We first prove the upper bound by observing from Lemma~\ref{Lemma:RiemannApproximation} and Theorem~\ref{Thm:FixedUpper} that 
 \begin{align*}
  \sup_{f_0\in\mathcal{F}_d\cap B_\infty(1)}\mathbb{E} \|\hat f_n -f_0\|_{L_2(Q)}^2 &\leq 2 \sup_{f_0\in\mathcal{F}_d\cap B_\infty(1)}\bigl\{n^{-1}\mathbb{E} \|\theta(\hat f_n) - \theta(f_0)\|_2^2 + \|f_0 - f(\theta(f_0))\|_{L_2(Q)}^2\bigr\} \\
  &\leq 2\sup_{\theta_0\in\mathcal{M}(\mathbb{L}_{d,n})\cap B_\infty(1)}\frac{1}{n}\mathbb{E}\|\hat\theta_n - \theta_0\|_2^2 + 8dn^{-1/d} \leq C_d n^{-1/d}\log^4 n,
 \end{align*}
 as desired.  Then by convexity of $\mathcal{H}$ and Proposition~\ref{Prop:FixedLower}, we have
 \begin{align*}
  \inf_{\tilde f_n}\sup_{f_0\in\mathcal{F}_d\cap B_\infty(1)} \mathbb{E}\|\tilde f_n - f_0\|_{L_2(Q)}^2 &\geq \inf_{\tilde f_n} \sup_{\theta_0\in\mathcal{M}(\mathbb{L}_{d,n}) \cap B_\infty(1)} \mathbb{E}\|\tilde f_n - f(\theta_0)\|_{L_2(Q)}^2 \\
  &= \inf_{\tilde f_n} \sup_{\theta_0\in\mathcal{M}(\mathbb{L}_{d,n}) \cap B_\infty(1)} \mathbb{E}\|f(\theta(\tilde f_n)) - f(\theta_0)\|_{L_2(Q)}^2 \\
  & = \inf_{\tilde \theta_n} \sup_{\theta_0\in\mathcal{M}(\mathbb{L}_{d,n}) \cap B_\infty(1)} \frac{1}{n}\mathbb{E}\|\tilde\theta_n - \theta_0\|_2^2\geq c_d n^{-1/d},
 \end{align*}
which completes the proof.
\end{proof}

\begin{proof}[Proof of Theorem~\ref{Thm:FixedOracle}]
Recall that the tangent cone at a point $x$ in a closed, convex set $K$ is defined as $T(x, K) := \{t(y-x): y\in K, t\geq 0\}$.  By \citet[Proposition~2.1]{Bellec2017} (see also \citet[Lemma~4.1]{ChatterjeeGuntuboyinaSen2017}), we have
 \begin{equation}
  R(\hat\theta_n, \theta_0) \leq \frac{1}{n}\inf_{\theta\in\mathcal{M}(\mathbb{L}_{d,n})}\Bigl\{\|\theta - \theta_0\|_2^2 + \delta\bigl(T(\theta, \mathcal{M}(\mathbb{L}_{d,n}))\bigr)\Bigr\}.\label{Eq:SharpOracle}
 \end{equation}
 For a fixed $\theta \in \mathcal{M}(\mathbb{L}_{d,n})$ such that $K(\theta) = K$, let $\mathbb{L}_{d,n} = \sqcup_{\ell=1}^K A_\ell$ be the partition of $\mathbb{L}_{d,n}$ into two-dimensional sheets $A_\ell$ such that $\theta$ is constant on each $A_\ell$. Define $m_\ell := |A_\ell|$. Then any $u \in T(\theta, \mathcal{M}(\mathbb{L}_{d,n}))$ must be isotonic when restricted to each of the two-dimensional sheets; in other words
 \[
  T(\theta, \mathcal{M}(\mathbb{L}_{d,n})) \subseteq \bigoplus_{\ell=1}^K T(0, \mathcal{M}(A_\ell)).
 \]
 By \citet[Proposition~3.1(9, 10)]{Amelunxenetal2014}, we have
\begin{equation}
\label{Eq:StatDimDecomp}
 \delta\bigl(T(\theta, \mathcal{M}(\mathbb{L}_{d,n}))\bigr) \leq \delta\biggl(\bigoplus_{\ell=1}^K T(0, \mathcal{M}(A_\ell))\biggr) = \sum_{\ell=1}^K \delta\bigl(T(0, \mathcal{M}(A_\ell))\bigr) = \sum_{\ell=1}^K \delta\bigl(\mathcal{M}(A_\ell)\bigr).
\end{equation}
By a consequence of the Gaussian Poincar\'{e} inequality \citep[cf.~][p.\ 73]{BoucheronLugosiMassart2013} and Proposition~\ref{Prop:GaussianWidthLattice}, we have
\begin{equation}
  \label{Eq:tmp7}
 \delta\bigl(\mathcal{M}(A_\ell)\bigr) \leq \biggl(\mathbb{E}\sup_{\theta\in\mathcal{M}(A_\ell)\cap B_2(1)} \langle \epsilon_{A_\ell}, \theta\rangle \biggr)^2 + 1 \lesssim \log_+^8 m_\ell.
\end{equation}
Thus, by~\eqref{Eq:StatDimDecomp}, \eqref{Eq:tmp7} and Lemma~\ref{Lemma:Jensen+} applied to $x\mapsto \log_+^8 x$, we have
\[
  \delta\bigl(T(\theta, \mathcal{M}(\mathbb{L}_{d,n}))\bigr) \lesssim \sum_{\ell=1}^K \log_+^8 m_\ell \lesssim K\log_+^8\biggl(\frac{n}{K}\biggr),
\]
which together with~\eqref{Eq:SharpOracle} proves the desired result.
\end{proof}

\begin{proof}[Proof of Theorem~\ref{Thm:FixedOracleBlock}]
For a fixed $\theta \in \mathcal{M}^{(k)}(\mathbb{L}_{d,n})$, let $\mathbb{L}_{d,n} = \sqcup_{\ell=1}^k R_\ell$ be the partition of $\mathbb{L}_{d,n}$ into hyperrectangles such that $\theta$ is constant on each hyperrectangle $R_\ell$. Suppose $R_\ell$ has side lengths $m_1,\ldots,m_d$ (so $|R_\ell| = \prod_{j=1}^d m_j$), so it can be partitioned into $\frac{|R_\ell|}{m_{j}m_{j'}}$ parallel two-dimensional sheets.  By choosing $m_j$ and $m_{j'}$ to be the largest two elements of the multiset $\{m_1,\ldots,m_d\}$ and using Jensen's inequality (noting that $x\mapsto x^{1-2/d}$ is concave when $d\geq 2$), we obtain
\begin{equation}
\label{Eq:tmp10}
  K(\theta) \leq \sum_{\ell=1}^k |R_\ell|^{1-2/d} \leq k \biggl(\frac{n}{k}\biggr)^{1-2/d}.
\end{equation}
This, combined with the oracle inequality in Theorem~\ref{Thm:FixedOracle}, gives the desired result.
\end{proof}

\begin{proof}[Proof of Proposition~\ref{Prop:FixedAdaptLower}]
Since the convex cone $\mathcal{M}(\mathbb{L}_{d,n})$ is invariant under translation by any $\theta_0\in\mathcal{M}^{(1)}(\mathbb{L}_{d,n})$, we may assume without loss of generality that $\theta_0 = 0$. By \citet[Corollary~1.2]{Chatterjee2014}, we have 
 \begin{equation}
 \label{Eq:ChatterjeeLowerBound}
  R(\hat\theta_n, 0) \geq \frac{1}{n} (t_0^2 - Ct_0^{3/2}),
 \end{equation}
 where 
 \begin{align*}
  t_0 :&= \argmax_{t \geq 0}\biggl\{\mathbb{E} \sup_{\theta\in\mathcal{M}(\mathbb{L}_{d,n}) \cap B_2(t)} \langle \epsilon, \theta\rangle - t^2/2\biggr\} \\
  &= \argmax_{t \geq 0} \biggl\{t\cdot \mathbb{E} \sup_{\theta\in\mathcal{M}(\mathbb{L}_{d,n})\cap B_2(1)} \langle \epsilon, \theta\rangle - t^2/2\biggr\} = \mathbb{E} \sup_{\theta\in\mathcal{M}(\mathbb{L}_{d,n})\cap B_2(1)}\langle \epsilon, \theta\rangle.
 \end{align*}
 By Proposition~\ref{Prop:GaussianWidthLattice}, we have 
 \[
  t_0 = \mathbb{E} \sup_{\theta\in\mathcal{M}(\mathbb{L}_{d,n})\cap B_2(1)}\langle \epsilon, \theta\rangle\geq c_d n^{1/2-1/d},
 \]
 which together with~\eqref{Eq:ChatterjeeLowerBound} proves the desired lower bound for cases $d\geq 3$. 

For the $d=2$ case, by Sudakov minorisation for Gaussian processes \citep[e.g.][Theorem~5.6 and the remark following it]{Pisier1999} and Lemma~\ref{Lemma:MetricEntropyLowerBound}, there exists a universal constant $\varepsilon_0 > 0$ such that
 \[
   \mathbb{E} \sup_{\theta\in\mathcal{M}(\mathbb{L}_{2,n})\cap B_2(1)} \langle \epsilon, \theta\rangle \gtrsim \varepsilon_0\log^{1/2} N\bigl(\varepsilon_0, \mathcal{M}(\mathbb{L}_{2,n}) \cap B_2(1),\|\cdot\|_2\bigr) \gtrsim \log n.
 \]
This, together with~\eqref{Eq:ChatterjeeLowerBound}, establishes the desired conclusion when $d = 2$.
\end{proof}

\begin{proof}[Proof of Corollary~\ref{Cor:d-2}]
  Without loss of generality, we may assume that $\theta_0\in\mathcal{M}_r(\mathbb{L}_{d,n})$ is a function of the final $r$ variables. For $x_3,\ldots,x_d \in \{1,\ldots, n_1\}$, we define the two-dimensional sheet $A_{x_3,\ldots,x_d} := \bigl\{(x_1,\ldots,x_d): x_1,x_2 \in\{1,\ldots,n_1\}\bigr\}$. When $r \leq d-2$, we have that $\theta_0$ is constant on each $A_{x_3,\ldots,x_d}$. Hence, by Theorem~\ref{Thm:FixedOracle}, 
  \[
    R(\hat\theta_n, \theta_0) \lesssim \frac{K(\theta_0)\log_+^8\bigl(n/K(\theta_0)\bigr)}{n} \lesssim n^{-2/d}\log^8 n. 
  \]
  Now suppose that $\theta_0\in\mathcal{M}_{d-1}(\mathbb{L}_{d,n})$. Let $m$ be a positive integer to be chosen later. Then $A_{x_3,\ldots,x_d} = \sqcup_{\ell = -m}^m A^{(\ell)}_{x_3,\ldots,x_d}$, where 
  \[
   A^{(\ell)}_{x_3,\ldots,x_d} := A_{x_3,\ldots,x_d}\cap \biggl\{v\in\mathbb{L}_{d,n}: \frac{\ell-1}{m} < (\theta_0)_v \leq \frac{\ell}{m}\biggr\}.
  \]
  Let $\theta^{(m)} \in \mathcal{M}(\mathbb{L}_{d,n})$ be the vector that takes the constant value $\ell/m$ on $A^{(\ell)}_{x_3,\ldots,x_d}$ for each $\ell = -m,\ldots,m$.  Then setting $m \asymp n^{2/(3d)}\log^{-8/3}n$, we have by Theorem~\ref{Thm:FixedOracle}  that
  \[
   R(\hat\theta_n, \theta_0) \lesssim \frac{\|\theta^{(m)}-\theta_0\|_2^2}{n} + \frac{K(\theta^{(m)})\log_+^8\bigl(n/K(\theta^{(m)})\bigr)}{n} \leq \frac{1}{m^2} + \frac{m}{n^{2/d}} \log^8 n\lesssim n^{-4/(3d)} \log^{16/3} n.
  \]
  as desired. 
  
  Finally, the $r=d$ case is covered in Theorem~\ref{Thm:FixedUpper}.
\end{proof}

\section{Proof of results in Section~\ref{Sec:RandomDesign}} 
\label{Sec:ProofRandom}
Henceforth we write $\mathbb{E}^X$ for the expectation conditional on $X_1,\ldots,X_n$, and also write $\mathbb{G}_n := n^{1/2}(\mathbb{P}_n-P)$.  The key ingredient in the proofs of both Theorems~\ref{Thm:RandomUpper} and~\ref{Thm:RandomOracle} is the following proposition, which controls the risk of the least squares estimator when $f_0 = 0$.  Recall that $\gamma_2 = 9/2$ and $\gamma_d = (d^2+d+1)/2$ for $d \geq 3$.
\begin{prop}
\label{Prop:RandomAdaptUpper}
  Let $d \geq 2$.  There exists a constant $C_{d,m_0,M_0}>0$, depending only on $d, m_0$ and $M_0$, such that
  \[
   R(\hat{f}_n,0) \leq C_{d,m_0,M_0} n^{-2/d} \log^{2\gamma_d} n.                                           
  \]
\end{prop}

The proof of Proposition~\ref{Prop:RandomAdaptUpper} requires several reduction techniques, which we detail in Section~\ref{Sec:ProofRandomAdaptUpper} below.  We first derive Theorems~\ref{Thm:RandomUpper} and~\ref{Thm:RandomOracle} from Proposition~\ref{Prop:RandomAdaptUpper}.

\begin{proof}[Proof of Theorem~\ref{Thm:RandomUpper}]
  Since the argument used in the proof of Theorem~\ref{Thm:FixedUpper}, up to~\eqref{Eq:t*}, does not depend on the design, we deduce from~\citet[][Corollary~1.2]{Chatterjee2014}, \citet[Proposition~3.1(5)]{Amelunxenetal2014} and the Cauchy--Schwarz inequality that
  \begin{equation}
  \label{Eq:tmp6}
    R(\hat{f}_n,f_0) \lesssim \frac{1}{n}\mathbb{E}\max\bigl\{1,\delta(\mathcal{M}(G_X)), n^{1/2}\delta(\mathcal{M}(G_X))^{1/2}\bigr\}.
  \end{equation}
  On the other hand, by Proposition~\ref{Prop:RandomAdaptUpper}, we have
  \begin{equation}
    \label{Eq:StatDim}
    \mathbb{E} \,\delta\bigl(\mathcal{M}(G_X)\bigr) \lesssim_{d,m_0,M_0} n^{1-2/d}\log^{2\gamma_d} n.
  \end{equation}
  We obtain the desired result by combining~\eqref{Eq:tmp6} and~\eqref{Eq:StatDim}.
\end{proof}

\begin{proof}[Proof of Theorem~\ref{Thm:RandomOracle}]
For any $f\in\mathcal{F}_d$, write $\theta_{f,X} := (f(X_1),\ldots,f(X_n))^\top \in \mathbb{R}^n$. By \citet[Proposition~2.1]{Bellec2017}, we have
\begin{align}
\label{Eq:tmp4}
 R(\hat{f}_n,f_0) &\leq \frac{1}{n} \mathbb{E}\,\biggl[\inf_{f\in\mathcal{F}_d} \Bigl\{\|\theta_{f,X} -\theta_{f_0,X}\|_2^2 + \delta\bigl(T(\theta_{f,X},\mathcal{M}(G_X))\bigr)\Bigr\}\biggr]\nonumber\\
& \leq \frac{1}{n} \inf_{k\in\mathbb{N}} \inf_{f\in\mathcal{F}_d^{(k)}}\Bigl\{ \mathbb{E} \|\theta_{f,X} -\theta_{f_0,X}\|_2^2 + \mathbb{E}\,\delta\bigl(T(\theta_{f,X},\mathcal{M}(G_X))\bigr)\Bigr\}.
\end{align}
Now, for a fixed $f\in\mathcal{F}_d^{(k)}$, let $R_1,\ldots,R_k$ be the corresponding hyperrectangles for which $f$ is constant when restricted to each $R_\ell$. Define $\mathcal{X}_\ell := R_\ell \cap \{X_1,\ldots,X_n\}$ and $N_\ell := |\mathcal{X}_\ell|$. Then for fixed $X_1,\ldots,X_n$, we have $T(\theta_{f,X}, \mathcal{M}(G_X)) \subseteq \bigoplus_{\ell=1}^k T\bigl(0, \mathcal{M}(G_{\mathcal{X}_\ell})\bigr) = \bigoplus_{\ell=1}^k \mathcal{M}(G_{\mathcal{X}_\ell})$. Hence by \citet[Proposition~3.1(9, 10)]{Amelunxenetal2014} and \eqref{Eq:StatDim}, we have that
\begin{align}
  \label{Eq:tmp5}
  \mathbb{E}\,\delta\bigl(T(\theta_{f,X},\mathcal{M}(G_X))\bigr) &= \mathbb{E}\Bigl[\mathbb{E}\Bigl\{\delta\bigl(T(\theta_{f,X}, \mathcal{M}(G_X))\bigr)\Bigm| N_1,\ldots,N_k\Bigr\}\Bigr]  \nonumber\\
  &\leq \mathbb{E}\biggl[\sum_{\ell=1}^k \mathbb{E}\Bigl\{\delta\bigl(\mathcal{M}(G_{\mathcal{X}_\ell})\bigr)\Bigm| N_\ell\Bigr\}\biggr] \nonumber\\
  &\lesssim_{d,m_0,M_0} \mathbb{E}\biggl\{ \sum_{\ell=1}^k N_\ell^{1-2/d} \log_+^{2\gamma_d} N_\ell \biggr\} \lesssim_{d} n  (k/n)^{2/d} \log_+^{2\gamma_d} (n/k),
\end{align}
where the final bound follows from applying Lemma~\ref{Lemma:Jensen+} to the function $x\mapsto x^{1-2/d} \log_+^{2\gamma_d}(x)$. We complete the proof by substituting~\eqref{Eq:tmp5} into~\eqref{Eq:tmp4} and observing that \[
\frac{1}{n}\inf_{f\in\mathcal{F}_{d}^{(k)}} \mathbb{E} \|\theta_{f,X}-\theta_{f_0,X}\|_2^2 = \inf_{f\in\mathcal{F}_d^{(k)}} \mathbb{E} \|f-f_0\|_{L_2(\mathbb{P}_n)}^2 = \inf_{f\in\mathcal{F}_d^{(k)}}\|f-f_0\|_{L_2(P)}^2,
\]
as required.
\end{proof}

\begin{proof}[Proof of Proposition~\ref{Prop:RandomLower}]
Without loss of generality, we may assume that $n = n_1^d$ for some $n_1\in\mathbb{N}$. Let $W := \{w \in \mathbb{L}_{d,n} : \sum_{j=1}^d w_j = n_1\}$. For any $w\in W$, define $\mathcal{C}_w := \prod_{j=1}^d \bigl(\frac{w_j-1}{n_1}, \frac{w_j}{n_1}\bigr]$.  Note that $x = (x_1,\ldots,x_d)^\top \in \cup_{w \in W} \mathcal{C}_{w}$ if and only if $\lceil n_1x_1\rceil + \cdots + \lceil n_1x_d\rceil = n_1$.  For any $\tau = (\tau_w) \in \{0,1\}^{|W|} =: T$, we define $f_\tau \in \mathcal{F}_d$ by
  \[
    f_\tau(x) := \begin{cases}
                   0 & \text{if $\lceil n_1x_1\rceil + \cdots + \lceil n_1x_d\rceil \leq n_1-1$}\\
                   1 & \text{if $\lceil n_1x_1\rceil + \cdots + \lceil n_1x_d\rceil \geq n_1+1$}\\
                   \rho\tau_{(\lceil n_1x_1\rceil,\ldots,\lceil n_1x_d\rceil)} & \text{if $x \in \cup_{w \in W} \mathcal{C}_{w}$},
                 \end{cases}
  \]
  where $\rho\in[0,1]$ is to be specified later. Moreover, let $\tau^w$ be the binary vector differing from $\tau$ in only the $w$ coordinate. We write $E_\tau$ for the expectation over $(X_1,Y_1),\ldots,(X_n,Y_n)$, where $Y_i = f_\tau(X_i) + \epsilon_i$ for $i=1,\ldots,n$. We let $E_X$ be the expectation over $(X_i)_{i=1}^n$ alone and $E_{Y|X,\tau}$ be the conditional expectation of $(Y_i)_{i=1}^n$ given $(X_i)_{i=1}^n$.  Given any estimator $\tilde f_n$, we have
  \begin{align}
    \max_{\tau\in T} E_\tau &\bigl\|\tilde f_n - f_\tau\bigr\|_{L_2(\mathbb{P}_n)}^2  \geq \frac{1}{2^{|W|}} \sum_{w\in W}\sum_{\tau\in T} E_\tau \int_{\mathcal{C}_w} (\tilde f_n - f_\tau)^2 \, \mathrm{d}\mathbb{P}_n\nonumber\\
    &= \frac{1}{2^{|W|+1}} \sum_{w\in W}\sum_{\tau\in T} \biggl\{E_\tau \int_{\mathcal{C}_w} (\tilde f_n - f_\tau)^2 \, \mathrm{d}\mathbb{P}_n + E_{\tau^w} \int_{\mathcal{C}_w} (\tilde f_n - f_{\tau^w})^2 \, \mathrm{d}\mathbb{P}_n\biggr\}\nonumber\\
    &=\frac{1}{2^{|W|+1}}\sum_{w\in W}\sum_{\tau\in T} E_X\biggl\{E_{Y|X,\tau} \int_{\mathcal{C}_w} (\tilde f_n - f_\tau)^2 \, \mathrm{d}\mathbb{P}_n + E_{Y|X,\tau^w} \int_{\mathcal{C}_w} (\tilde f_n - f_{\tau^w})^2 \, \mathrm{d}\mathbb{P}_n\biggr\}\nonumber\\
    &\geq \frac{1}{2^{|W|+1}}\sum_{w\in W}\sum_{\tau\in T} E_X\biggl\{ \frac{1}{4}\int_{\mathcal{C}_w}(f_\tau-f_{\tau^w})^2\,\mathrm{d}\mathbb{P}_n\Bigl[1-d_{\mathrm{TV}}\Bigl(P_{Y|X,\tau}, P_{Y|X,\tau^w}\Bigr)\Bigr]\biggr\},\label{Eq:Assouad}
  \end{align}
  where $P_{Y|X,\tau}$ (respectively $P_{Y|X,\tau^w}$) is the conditional distribution of $(Y_i)_{i=1}^n$ given $(X_i)_{i=1}^n$ when the true signal is $f_\tau$ (respectively $f_{\tau^w}$). The final inequality in the above display follows because for $\Delta := \bigl(\int_{\mathcal{C}_w} (f_\tau - f_{\tau^w})^2 \, \mathrm{d}\mathbb{P}_n\bigr)^{1/2}$ and $A := \bigl\{\int_{\mathcal{C}_w} (\tilde{f}_n - f_\tau)^2 \, \mathrm{d}\mathbb{P}_n \geq \Delta^2/4\bigr\}$, we have
  \begin{align*}
    E_{Y|X,\tau}\int_{\mathcal{C}_w} (\tilde f_n - f_\tau)^2 \, \mathrm{d}\mathbb{P}_n + E_{Y|X,\tau^w} \int_{\mathcal{C}_w} (\tilde f_n - f_{\tau^w})^2 \, \mathrm{d}\mathbb{P}_n &\geq \frac{\Delta^2}{4}\bigl\{P_{Y|X,\tau}(A) + P_{Y|X,\tau^w}(A^c)\Bigr\}\\
    &\geq \frac{\Delta^2}{4}\bigl\{1-d_{\mathrm{TV}}\bigl(P_{Y|X,\tau}, P_{Y|X,\tau^w}\bigr)\bigr\}.
  \end{align*}
By Pinsker's inequality \citep[cf.~][p.~62]{Pollard2002}, we obtain that
  \begin{equation}
    \label{Eq:TV}
    d_{\mathrm{TV}}^2\Bigl(P_{Y|X,\tau}, P_{Y|X,\tau^w}\Bigr) \leq \frac{1}{2} d_{\mathrm{KL}}^2(P_{Y|X,\tau},P_{Y|X,\tau^w}) = \frac{n}{4}\|f_{\tau} - f_{\tau^w}\|_{L_2(\mathbb{P}_n)}^2.
  \end{equation}
Writing $N_w := \sum_{i=1}^n \mathbbm{1}_{\{X_i \in \mathcal{C}_w\}}$, we have $N_w\sim \mathrm{Bin}(n, P(\mathcal{C}_w))$, so $E_X N_w \geq m_0$ and $E_X N_w^{3/2} \leq (E_X N_w^2 \, E_X N_w)^{1/2} \leq 2^{1/2}M_0^{3/2}$. Thus, together with~\eqref{Eq:TV}, we have
  \begin{align*}
    E_X\biggl\{\int_{\mathcal{C}_w}(f_\tau-f_{\tau^w})^2\,\mathrm{d}\mathbb{P}_n\Bigl[1-& d_{\mathrm{TV}}\Bigl(P_{Y|X,\tau}, P_{Y|X,\tau^w}\Bigr)\Bigr]\biggr\}\\
    &\geq E_X \biggl\{\|f_{\tau} - f_{\tau^w}\|_{L_2(\mathbb{P}_n)}^2\biggl( 1 - \frac{n^{1/2}}{2}\|f_{\tau} - f_{\tau^w}\|_{L_2(\mathbb{P}_n)} \biggr)\biggr\}\\
    & =\frac{\rho^2}{n} E_X N_w -\frac{\rho^3}{2n}E_X  N_w^{3/2} \geq \frac{\rho^2}{n}\biggl(m_0-\frac{\rho}{2^{1/2}} M_0^{3/2}\biggr).
  \end{align*}
  Substituting the above inequality into~\eqref{Eq:Assouad}, we obtain that for $\rho = 2^{3/2}m_0/(3M_0^{3/2})$, 
  \[
    \max_{\tau\in T} E_\tau \bigl\|\tilde f_n - f_\tau\bigr\|_{L_2(\mathbb{P}_n)}^2 \geq \frac{|W|}{27n}\frac{m_0^3}{M_0^3} \geq c_{d,m_0,M_0} n^{-1/d},
  \]
  where the final inequality follows from a counting argument as in~\eqref{Eq:DiagonalCardinality}.  This completes the proof.
\end{proof}

\begin{proof}[Proof of Proposition~\ref{Prop:RandomAdaptLower}]
  Clearly we only need to establish the claim for $f_0 = 0$. By Lemma~\ref{Lemma:RandomAntichain}, there is an event $\mathcal{E}$ with probability at least $1- e^{-ed^{-1} (M_0 n)^{1/d}\log (M_0 n)}$ on which the data points $X_1,\ldots,X_n$ contain a maximal antichain $W_X$ with cardinality at least $n^{1-1/d}/(2eM_0^{1/d})$. Write $W_X^+ := \{X_i : \exists w \in W_X, X_i \succ w\}$ and $W_X^- := \{X_i : \exists w \in W_X, X_i \prec w\}$. For each realisation of the $n$-dimensional Gaussian random vector $\epsilon$, we define $\theta_X = \theta_X(\epsilon) = ((\theta_X)_w)$ by 
  \[
    (\theta_X)_w := \begin{cases}
                 1 & \text{ if $w \in W_X^+$}\\
                 \mathrm{sgn}(\epsilon_w) & \text{ if $w \in W_X$}\\
                 -1 & \text{ if $w \in W_X^-$}.
               \end{cases}
  \]
  We see that $\theta_X \in \mathcal{M}(G_X)$. By \citet[Theorem~1.1]{Chatterjee2014}, for $f_0 = 0$, we have that
  \[
    n^{1/2} \bigl\|\hat f_n\bigr\|_{L_2(\mathbb{P}_n)} = \argmax_{t\geq 0}\biggl(\sup_{\theta \in\mathcal{M}(G_X)\cap B_2(t)} \langle \epsilon, \theta \rangle - t^2/2\biggr) = \sup_{\theta \in\mathcal{M}(G_X)\cap B_2(1)} \langle \epsilon, \theta \rangle.
  \]
  Hence
  \begin{align}
    \mathbb{E}\bigl\|\hat f_n\bigr\|_{L_2(\mathbb{P}_n)} &= \frac{1}{n^{1/2}}\mathbb{E}\sup_{\theta \in\mathcal{M}(G_X)\cap B_2(1)} \langle \epsilon, \theta \rangle \geq \frac{1}{n^{1/2}}\mathbb{E}\biggl(\biggl\langle \epsilon, \frac{\theta_X(\epsilon)}{\|\theta_X(\epsilon)\|_2}\biggr\rangle \mathbbm{1}_\mathcal{E}\biggr) \nonumber\\
    & \geq \frac{1}{n}\mathbb{E}\biggl(\sum_{i: X_i\in W_X^+} \epsilon_i \mathbbm{1}_\mathcal{E} - \sum_{i: X_i\in W_X^-} \epsilon_i \mathbbm{1}_\mathcal{E} + \sum_{i: X_i\in W_X} |\epsilon_i| \mathbbm{1}_\mathcal{E}\biggr).\label{Eq:tmp8}
  \end{align}
  The first two terms in the bracket are seen to be zero by computing the expectation conditionally on $X_1,\ldots,X_n$. For the third term, we have that
  \begin{equation}
    \label{Eq:tmp9}
    \mathbb{E}\biggl(\sum_{i: X_i\in W_X} |\epsilon_i| \mathbbm{1}_\mathcal{E}\biggr) = \mathbb{E}\biggl\{\sum_{i: X_i\in W_X} \mathbb{E}^X \bigl(|\epsilon_i| \mathbbm{1}_\mathcal{E}\bigr) \biggr\} \geq \Bigl(\frac{2}{\pi}\Bigr)^{1/2}\mathbb{E} \bigl(|W_X|\mathbbm{1}_\mathcal{E}\bigr) \gtrsim_{d,M_0} n^{1-1/d},
  \end{equation}
  where the final inequality follows from Lemma~\ref{Lemma:RandomAntichain}. By~\eqref{Eq:tmp8},~\eqref{Eq:tmp9} and the Cauchy--Schwarz inequality, we have that
  \[
    \mathbb{E} \bigl\|\hat f_n\bigr\|_{L_2(\mathbb{P}_n)}^2 \geq \bigl\{\mathbb{E} \bigl\|\hat f_n\bigr\|_{L_2(\mathbb{P}_n)}\bigr\}^2 \gtrsim_{d,M_0} n^{-2/d},
  \]
  as desired.
\end{proof}

\subsection{Proof of Proposition~\ref{Prop:RandomAdaptUpper}}
\label{Sec:ProofRandomAdaptUpper}

The proof of Proposition~\ref{Prop:RandomAdaptUpper} is rather technical, so we sketch a brief outline of the main steps below:
\begin{enumerate}[label={\textbf{Step \arabic*}.}, leftmargin=*, align=left]
\item Instead of bounding $R(\hat{f}_n,0)$ directly, we first consider bounding
\begin{equation}
  \label{Eq:Step1}
  \mathbb{E} \bigl\{\| \hat f_n\|_{L_2(P)}^2 \mathbbm{1}_{\{\|\hat f_n\|_\infty \leq 4\log^{1/2} n\}}\bigr\}.
\end{equation}
By Proposition~\ref{Prop:Reduction}, this task essentially reduces to understanding two empirical processes~\eqref{Eq:GaussianMultiplier} and~\eqref{Eq:RademacherQuadratic}.  By means of Lemmas~\ref{Lemma:Conversion} and~\ref{Lemma:Contraction}, this in turn reduces to the study of the symmetrised local empirical process
\begin{equation}
  \label{Eq:RademacherMultiplier}
  \mathbb{E} \sup_{f\in\mathcal{F}_d\cap B_\infty(1)\cap B_2(r,P)}\biggl|\frac{1}{n^{1/2}}\sum_{i=1}^n \xi_i f(X_i)\biggr|,
\end{equation}
for a suitable $L_2(P)$ radius $r$.

\item To obtain a sharp bound on the empirical process in~\eqref{Eq:RademacherMultiplier}, which constitutes the main technical challenge of the proof, we slice $[0,1]^d$ into strips of the form $[0,1]^{d-1} \times [\frac{\ell-1}{n_1}, \frac{\ell}{n_1}]$, for $\ell = 1,\ldots,n_1$, and decompose  $\sum_{i=1}^n\xi_i f(X_i)$ into sums of smaller empirical processes over these strips. Each of these smaller empirical processes is then controlled via a bracketing entropy chaining argument (Lemma~\ref{Lemma:LocalMaximalInequality}).  The advantage of this decomposition is that the block monotonicity permits good control of the $L_2(P)$ norm of the envelope function in each strip (Lemma~\ref{Lemma:EnvelopeIntegral}).

\item Having bounded~\eqref{Eq:RademacherMultiplier}, and hence also~\eqref{Eq:Step1}, we finally translate the bound of~\eqref{Eq:Step1} back to a bound for $R(\hat{f}_n,0)$, our original quantity of interest. The cost of $L_\infty$ truncation is handled in Lemma~\ref{Lemma:Linfty}, whereas our understanding of the symmetrised empirical process in~\eqref{Eq:RademacherMultiplier} helps to control the discrepancy between the $L_2(P)$ norm and $L_2(\mathbb{P}_n)$ norm risks (cf.\ Proposition~\ref{Prop:RatioEmpiricalProcess}). 
\end{enumerate}

In our empirical process theory arguments, since our least squares estimator $\hat{f}_n$ is defined to be lower semi-continuous, we can avoid measurability and countability digressions by defining $\mathcal{G}$ the class of real-valued lower semi-continuous functions on $[0,1]^d$ and $\mathcal{F}_d':=\{f\in\mathcal{F}_d\cap\mathcal{G}: f|_{(\mathbb{Q}\cap[0,1])^d} \subseteq \mathbb{Q}\}$. This is a countable, uniformly dense\footnote{Here `uniformly dense' means that for any $f \in \mathcal{F}_d\cap \mathcal{G}$, we can find a sequence $(f_m)$ in $\mathcal{F}_d'$ such that $\|f_m-f\|_\infty \to 0$. This can be done by defining, e.g., $f_m(x) := m^{-1}\lceil mf(x)\rceil$.} subset of $\mathcal{F}_d\cap\mathcal{G}$ so that, for example, $\sup_{f\in\mathcal{F}_{d}\cap\mathcal{G}} \mathbb{G}_n f= \sup_{f\in\mathcal{F}_{d}'} \mathbb{G}_n f$. 

The main content of Step 1 is the following proposition.
\begin{prop}
  \label{Prop:Reduction}
  Suppose that for each $n \in \mathbb{N}$ there exist a function $\phi_n: [0,\infty)\to [0,\infty)$ and $r_n \geq n^{-1/2}\log^{1/2}n$ such that $\phi_n(r_n) \leq n^{1/2}r_n^2$.  Moreover, assume that for all $r \geq r_n$ the map $r\mapsto \phi_n(r)/r$ is non-increasing and 
  \begin{equation}
  \label{Eq:GaussianMultiplier}
  \mathbb{E} \sup_{f\in\mathcal{F}_d\cap B_\infty(4\log^{1/2} n)\cap B_2(r, P)} \biggl|\frac{1}{n^{1/2}}\sum_{i=1}^n\epsilon_i f(X_i)\biggr| \leq C_1 \phi_n(r),
  \end{equation}
  and
  \begin{equation}
    \label{Eq:RademacherQuadratic}
    \mathbb{E} \sup_{f\in\mathcal{F}_d\cap B_\infty(4\log^{1/2} n)\cap B_2(r, P)} \biggl|\frac{1}{n^{1/2}}\sum_{i=1}^n\xi_i f^2(X_i)\biggr| \leq C_2 \phi_n(r),
  \end{equation}
  for some constants $C_1, C_2 > 0$ that do not depend on $r$  and $n$. Then for $f_0 = 0$, we have that
  \[
    \mathbb{E}\bigl\{\|\hat f_n\|_{L_2(P)}^2 \mathbbm{1}_{\{\|\hat f_n\|_\infty\leq 4\log^{1/2} n\}}\bigr\} \lesssim_{C_1,C_2} r_n^2.
  \]
\end{prop}
\begin{proof} 
Let $\mathbb{M}_n f := \frac{2}{n}\sum_{i=1}^n \epsilon_if(X_i) - \frac{1}{n}\sum_{i=1}^n f^2(X_i)$ and $M f := \mathbb{E} (\mathbb{M}_n f) = -\|f\|_{L_2(P)}^2 =: -Pf^2$. Then
\[
  |\mathbb{M}_n f - M f| \leq \biggl|\frac{2}{n}\sum_{i=1}^n \epsilon_if(X_i)\biggr| + \bigl|(\mathbb{P}_n - P) f^2\bigr|.
\]
Moreover, by definition of $\hat{f}_n$, we have $\sum_{i=1}^n \{\epsilon_i - \hat f_n(X_i)\}^2 \leq \sum_{i=1}^n \epsilon_i^2$, so $\mathbb{M}_n \hat f_n \geq 0$. Fix $s \geq 1$ and for $\ell \in\mathbb{N}$, let $\mathcal{F}_{d,\ell} := \mathcal{F}_d\cap B_\infty(4\log^{1/2} n)\cap B_2(2^\ell sr_n)$. 	
Then by a union bound, we have
\begin{align}
  \label{Eq:prop8tmp1}
  &\mathbb{P}\Bigl(\bigl\{\|\hat f_n\|_{L_2(P)} \geq s r_n\bigr\} \cap \bigl\{\|\hat f_n\|_\infty \leq 4\log^{1/2} n\bigr\}\Bigr) \leq \sum_{\ell=1}^\infty \mathbb{P}\biggl(\sup_{f\in\mathcal{F}_{d,\ell} \setminus \mathcal{F}_{d,\ell-1}} \mathbb{M}_n f \geq 0\biggr) \nonumber \\
  &\leq \sum_{\ell=1}^\infty \mathbb{P}\biggl(\sup_{f\in\mathcal{F}_{d,\ell}} \bigl(\mathbb{M}_n f - M f\bigr)\geq 2^{2\ell-2}s^2 r_n^2\biggr)\nonumber\\
  &\leq \sum_{\ell=1}^\infty \mathbb{P}\biggl(\sup_{f\in \mathcal{F}_{d,\ell}}\biggl|\frac{1}{n^{1/2}}\sum_{i=1}^n\epsilon_i f(X_i)\biggr| \geq 2^{2\ell-4} s^2n^{1/2}r_n^2\biggr)
  + \sum_{\ell=1}^\infty \mathbb{P}\biggl(\sup_{f\in \mathcal{F}_{d,\ell}}\Bigl|\mathbb{G}_n f^2 \Bigr| \geq 2^{2\ell-3} s^2n^{1/2} r_n^2\biggr).
\end{align}
By a moment inequality for empirical processes \citep[Proposition~3.1]{GineLatalaZinn2000} and~\eqref{Eq:GaussianMultiplier}, we have
\begin{align}
\label{Eq:prop8tmp3}
  \mathbb{E}\biggl\{\biggl(\sup_{f\in \mathcal{F}_{d,\ell}}\biggl|\frac{1}{n^{1/2}}\sum_{i=1}^n\epsilon_i f(X_i)\biggr|\biggr)^4\biggr\} &\lesssim_{C_1} \phi_n^4(2^\ell s r_n) +(2^\ell sr_n)^4+ \frac{\log^4 n}{n^2}.
\end{align}
Similarly, by symmetrisation (cf.\ \citet[Lemma~2.3.1]{vanderVaartWellner1996}), the moment inequality for empirical processes mentioned above and condition~\eqref{Eq:RademacherQuadratic}, we have 
\begin{equation}
 \mathbb{E} \biggl\{\biggl(\sup_{f\in \mathcal{F}_{d,\ell}} \bigl|\mathbb{G}_n f^2\bigr|\biggr)^4\biggr\} \lesssim
 \mathbb{E} \biggl\{\biggl(\sup_{f\in \mathcal{F}_{d,\ell}}\biggl|\frac{1}{n^{1/2}}\sum_{i=1}^n\xi_i f^2(X_i)\biggr|\biggr)^4\biggr\} \lesssim_{C_2} \phi_n^4(2^\ell s r_n) + (2^\ell sr_n)^4+\frac{\log^2 n}{n^2}.\label{Eq:prop8tmp4}
\end{equation}
By~\eqref{Eq:prop8tmp1}, \eqref{Eq:prop8tmp3}, \eqref{Eq:prop8tmp4} and Markov's inequality, we obtain that
\begin{align}
  \label{Eq:prop8tmp5}
  \mathbb{P}\Bigl(\bigl\{\|\hat f_n\|_{L_2(P)} \geq s r_n\bigr\} \cap \bigl\{\|\hat f_n\|_\infty \leq 4\log^{1/2} n\bigr\}\Bigr) \lesssim_{C_1,C_2} \sum_{\ell=1}^\infty \biggl(\frac{\phi_n(2^\ell s r_n)}{2^{2\ell} s^2n^{1/2}r_n^2}\biggr)^4 + \frac{1}{s^4} \lesssim \frac{1}{s^4},
\end{align}
where we have used the assumption $r_n\geq n^{-1/2} \log^{1/2} n$ and the fact that $\phi_n(2^\ell s r_n)\leq 2^\ell s\phi_n(r_n) \leq 2^\ell s n^{1/2}r_n^2$ for the non-increasing function $r\mapsto \phi_n(r)/r$. The bound in~\eqref{Eq:prop8tmp5} is valid for all $s \geq 1$. Hence
\begin{align*}
  \mathbb{E}\bigl\{\|\hat f_n\|_{L_2(P)}^2 \mathbbm{1}_{\{\|\hat f_n\|_\infty\leq 4\log^{1/2} n\}}\bigr\} &= \int_{0}^\infty \mathbb{P}\bigl\{\bigl\|\hat f_n\bigr\|_{L_2(P)}^2\mathbbm{1}_{\{\|\hat f_n\|_\infty\leq 4\log^{1/2} n\}} \geq t\bigr\}\ \mathrm{d}t\\
  &\leq r_n^2 + 2r_n^2 \int_{1}^\infty s \, \mathbb{P}\bigl\{\bigl\|\hat f_n\bigr\|_{L_2(P)}^2\mathbbm{1}_{\{\|\hat f_n\|_\infty\leq 4\log^{1/2} n\}} \geq s^2r_n^2\bigr\}\ \mathrm{d}s\\
  &\lesssim_{C_1,C_2} r_n^2,
\end{align*}
as desired.
\end{proof}

The proposition below on the size of the symmetrised empirical process solves Step 2 in the outline of the proof of Proposition~\ref{Prop:RandomAdaptUpper}.
\begin{prop}
\label{Prop:SymmetrisedProcess}
Fix $d\geq 2$ and suppose that $r \geq n^{-\max\{1/d,(1-2/d)\}}\log^{(d^2-d)/2} n$. There exists a constant $C_{d,m_0,M_0}>0$, depending only on $d, m_0$ and $M_0$, such that
\[
  \mathbb{E} \sup_{f\in\mathcal{F}_d\cap B_\infty(1)\cap B_2(r, P)} \biggl|\frac{1}{n^{1/2}}\sum_{i=1}^n \xi_i f(X_i)\biggr| \leq C_{d,m_0,M_0} r n^{1/2-1/d} \log^{\gamma_d - 1/2} n.
\]
\end{prop}

\begin{proof}
 It is convenient here to work with the class of block decreasing functions $\mathcal{F}_{d,\downarrow} := \{f:[0,1]^d\to\mathbb{R}: -f\in \mathcal{F}_d\}$ instead. We write $\mathcal{F}_d^+ := \{f\in\mathcal{F}_d: f\geq 0\}$ and $\mathcal{F}_{d,\downarrow}^+ := \{f\in\mathcal{F}_{d,\downarrow} : f\geq 0\}$. By replacing $f$ with $-f$ and decomposing any function $f$ into its positive and negative parts, it suffices to prove the result with $\mathcal{F}_{d,\downarrow}^+$ in place of $\mathcal{F}_d$. We handle the cases $d = 2$ and $d\geq 3$ separately. 
 
 \vspace{0.5cm}\noindent\underline{Case $d = 2$.}  
 We apply Lemma~\ref{Lemma:LocalMaximalInequality} with $\eta = r/(2n)$ and Lemma~\ref{Lemma:BracketingEntropyBound} to obtain 
 \begin{align*}
 \mathbb{E}\sup_{f\in\mathcal{F}_{2,\downarrow}^+ \cap B_\infty(1)\cap B_2(r, P)}\biggl|\frac{1}{n^{1/2}}\sum_{i=1}^n\xi_i f(X_i)\biggr| & \lesssim_{d,m_0,M_0} n^{1/2}\eta + \log^3 n \int_{\eta}^{r} \frac{r}{\varepsilon} \,\mathrm{d}\varepsilon +  \frac{\log^4 n (\log\log n)^2}{n^{1/2}}\\
 &\lesssim  r\log^4 n,
 \end{align*}
 as desired.
 
 \vspace{0.5cm}\noindent\underline{Case $d \geq 3$.} We assume without loss of generality that $n = n_1^d$ for some $n_1\in\mathbb{N}$.  We define strips $I_\ell := [0,1]^{d-1}\times [\frac{\ell-1}{n_1}, \frac{\ell}{n_1}]$ for $\ell = 1,\ldots,  n_1$, so that $[0,1]^d = \cup_{\ell=1}^{n_1} I_\ell$.  Our strategy is to analyse the expected supremum of the symmetrised empirical process when restricted to each strip.  To this end, define $S_\ell := \{X_1,\ldots,X_n\}\cap I_\ell$ and $N_\ell := |S_\ell|$, and let $\Omega_0 := \{m_0n^{1-1/d}/2  \leq \min_\ell N_\ell \leq \max_\ell  N_\ell \leq 2M_0n^{1-1/d}\}$.  Then by Hoeffding's inequality, 

 \[
  \mathbb{P}(\Omega_0^c) \leq \sum_{\ell=1}^{n_1} \mathbb{P}\biggl(\Bigl|N_\ell - \mathbb{E} N_\ell \Bigr| > \frac{m_0n}{2n_1}\biggr) \leq 2n_1\exp\bigl(-m_0^2n^{1-2/d}/8\bigr).
 \]
 Hence we have
 \begin{equation}
  \mathbb{E}\sup_{f\in\mathcal{F}_{d,\downarrow}^+\cap B_\infty(1)\cap B_2(r,P)}\Bigl|\frac{1}{n^{1/2}}\sum_{i=1}^n\xi_i f(X_i)\Bigr| \leq \sum_{\ell=1}^{n_1}\mathbb{E}\biggl( \frac{N_\ell^{1/2}}{n^{1/2}} E_\ell \,\mathbbm{1}_{\Omega_0}\biggr) + C \exp\bigl(-m_0^2 n^{1-2/d}/16\bigr),
  \label{Eq:ConditionalExpectation}
 \end{equation}
 where
 \[
  E_\ell :=  \mathbb{E}\biggl\{ \sup_{f\in\mathcal{F}_{d,\downarrow}^+\cap B_\infty(1)\cap B_2(r,P)}\Bigl|\frac{1}{N_\ell^{1/2}}\sum_{i:X_i\in S_\ell}\xi_i f(X_i)\Bigr| \biggm | N_1,\ldots,N_{n_1}\biggr\}.
 \]
 By Lemma~\ref{Lemma:EnvelopeIntegral}, for any $f\in\mathcal{F}_{d,\downarrow}^+\cap B_\infty(1)\cap B_2(r,P)$ and $\ell \in \{1,\ldots,n_1\}$, we have $\int_{I_\ell} f^2\,\mathrm{d}P \leq 7(M_0/m_0)\ell^{-1} r^2 \log^d n =: r_{n,\ell}^2$.  Consequently, we have by Lemma~\ref{Lemma:LocalMaximalInequality} that for any $\eta\in[0,r_{n,\ell}/3)$,
 \begin{equation}
 \label{Eq:E_l}
  E_\ell \lesssim N_\ell^{1/2}\eta + \int_{\eta}^{r_{n,\ell}} H^{1/2}_{[\,]}(\varepsilon)\,\mathrm{d}\varepsilon + \frac{H_{[\,]}(r_{n,\ell})}{N_\ell^{1/2}},
 \end{equation}
 where $H_{[\,]}(\varepsilon) := \log N_{[\,]}\bigl(\varepsilon, \mathcal{F}_{d,\downarrow}^+(I_\ell)\cap B_\infty(1;I_\ell)\cap B_2(r_{n,\ell}, P; I_\ell), \|\cdot\|_{L_2(P;I_\ell)}\bigr)$.  Here, $\|f\|_{L_2(P;I_\ell)}^2 := \int_{I_\ell} f^2 \, dP$, the set $\mathcal{F}_{d,\downarrow}^+(I_\ell)$ is the class of non-negative functions on $I_\ell$ that are block decreasing, $B_\infty(1;I_\ell)$ is the class of functions on $I_\ell$ that are bounded by $1$ and $B_2(r_{n,\ell},P; I_\ell)$ is the class of measurable functions $f$ on $I_\ell$ with $\|f\|_{L_2(P;I_\ell)} \leq r_{n,\ell}$. Any $g\in \mathcal{F}_{d,\downarrow}^+(I_\ell)\cap B_\infty(1;I_\ell)\cap B_2(r_{n,\ell},P; I_\ell)$ can be rescaled into a function $f_g\in\mathcal{F}_{d,\downarrow}^+\cap B_\infty(1)\cap B_2\bigl(n_1^{1/2}(M_0/m_0)^{1/2}r_{n,\ell},P\bigr)$ via the invertible map $f_g(x_1,\ldots,x_{d-1},x_d) := g(x_1,\ldots,x_{d-1},(x_d+\ell-1)/n_1)$. Moreover, $\int_{[0,1]^d} (f_g-f_{g'})^2\,\mathrm{d}P \geq n_1(m_0/M_0)\int_{I_\ell} (g-g')^2\,\mathrm{d}P$. Thus, by Lemma~\ref{Lemma:BracketingEntropyBound}, for $\epsilon \in [\eta,r_{n,\ell}]$,
 \begin{align*}
  H_{[\,]}(\varepsilon) &\leq \log N_{[\,]}\bigl(n^{1/(2d)}(m_0/M_0)^{1/2}\varepsilon, \mathcal{F}_{d,\downarrow}^+\cap B_\infty(1)\cap B_2\bigl(n^{1/(2d)}(M_0/m_0)^{1/2}r_{n,\ell}\bigr), \|\cdot\|_{L_2(P)}\bigr) \\
  &\lesssim_{d,m_0,M_0} \biggl(\frac{r_{n,\ell}}{\varepsilon}\biggr)^{2(d-1)}\log_+^{d^2}(1/\epsilon).
 \end{align*}
 Substituting the above entropy bound into~\eqref{Eq:E_l}, and choosing $\eta = n^{-1/(2d)} r_{n,\ell}$, we obtain
 \[
  E_\ell \lesssim_{d,m_0,M_0} N_\ell^{1/2} \eta + \log^{d^2/2} n \int_{\eta}^{r_{n,\ell}} \biggl(\frac{r_{n,\ell}}{\varepsilon}\biggr)^{d-1}\,\mathrm{d}\varepsilon + \frac{\log^{d^2} n}{N_\ell^{1/2}} \lesssim N_\ell^{1/2}\eta + \frac{r_{n,\ell}^{d-1}\log^{d^2/2} n}{\eta^{d-2}}+\frac{\log^{d^2} n}{N_\ell^{1/2}}.
\]
Hence
\begin{equation}
\label{Eq:E_l2}
  E_\ell \mathbbm{1}_{\Omega_0} \lesssim_{d,m_0,M_0} r_{n,\ell}\, n^{1/2-1/d} \log^{d^2/2} n + n^{-1/2+1/(2d)} \log^{d^2} n \lesssim_{m_0,M_0} r_{n,\ell}\, n^{1/2-1/d} \log^{d^2/2} n,
 \end{equation}
 where in the final inequality we used the conditions that $d\geq 3$ and $r \geq n^{-(1-2/d)}\log^{(d^2-d)/2} n$. Combining~\eqref{Eq:ConditionalExpectation} and~\eqref{Eq:E_l2}, we have that
 \begin{align*}
  \mathbb{E}\sup_{f\in\mathcal{F}_{d,\downarrow}^+\cap B_\infty(1)\cap B_2(r,P)}\biggl|\frac{1}{n^{1/2}}\sum_{i=1}^n\xi_i f(X_i)\biggr| &\lesssim_{d,m_0,M_0} r n^{1/2-1/d}  \log^{(d^2+d)/2} n \, \biggl(\frac{1}{n_1^{1/2}} \sum_{\ell=1}^{n_1} \ell^{-1/2}\biggr)\\
  &\lesssim r n^{1/2-1/d} \log^{(d^2+d)/2} n,
 \end{align*}
 which completes the proof.
\end{proof}

Finally, we need the following proposition to switch between $L_2(P)$ and $L_2(\mathbb{P}_n)$ norms as described in Step 3. 
\begin{prop}
\label{Prop:RatioEmpiricalProcess}
Fix $d\geq 2$ and suppose that $f_0 = 0$. There exists a constant $C_{d,m_0,M_0}>0$, depending only on $d$, $m_0$ and $M_0$, such that  
\[
  \mathbb{E}\bigl\{\bigl\|\hat f_n\bigr\|_{L_2(\mathbb{P}_n)}^2 \mathbbm{1}_{\{\|\hat f_n\|_\infty\leq 4\log^{1/2} n\}}\bigr\} \leq C_{d,m_0,M_0}\Bigl[n^{-2/d}\log^{2\gamma_d}n + \bigl\{\mathbb{E}\bigl\|\hat f_n\bigr\|_{L_2(P)}^2 \mathbbm{1}_{\{\|\hat f_n\|_\infty\leq 4\log^{1/2} n\}}\bigr\}\Bigr].
\]
\end{prop}

\begin{proof}
To simplify notation, we define $\tilde f_n := \hat f_n \mathbbm{1}_{\{\|\hat f_n\|_\infty \leq 4\log^{1/2} n\}}$ and $r_n := n^{-1/d}\log^{\gamma_d} n$. We write
\begin{align}
  \label{Eq:Decomposition}
  \mathbb{E}\bigl\{\bigl\|\hat f_n\bigr\|_{L_2(\mathbb{P}_n)}^2 &\mathbbm{1}_{\{\|\hat f_n\|_\infty\leq 4\log^{1/2} n\}}\bigr\} = \mathbb{E}\bigl\|\tilde f_n\bigr\|_{L_2(\mathbb{P}_n)}^2 \nonumber \\
&\hspace{1cm}= \mathbb{E}\bigl\{\bigl\|\tilde f_n\bigr\|_{L_2(\mathbb{P}_n)}^2 \mathbbm{1}_{\{\|\hat f_n\|_{L_2(P)} \leq r_n\}}\bigr\} + \mathbb{E}\bigl\{\bigl\|\tilde f_n\bigr\|_{L_2(\mathbb{P}_n)}^2 \mathbbm{1}_{\{\|\hat f_n\|_{L_2(P)} > r_n\}}\bigr\},
\end{align}
and control the two terms on the right hand side of~\eqref{Eq:Decomposition} separately. For the first term, we have
\begin{align}
  \mathbb{E}\bigl\{\bigl\|\tilde f_n\bigr\|_{L_2(\mathbb{P}_n)}^2 \mathbbm{1}_{\{\|\hat f_n\|_{L_2(P)} \leq r_n\}}\bigr\}
  &\leq \mathbb{E}\sup_{f\in\mathcal{F}_d\cap B_\infty(4\log^{1/2} n)\cap B_2(r_n,P)}\frac{1}{n}\sum_{i=1}^n f^2(X_i)\nonumber\\
  &\lesssim r_n^2 + \frac{1}{n}\mathbb{E}\sup_{f\in\mathcal{F}_d\cap B_\infty(4\log^{1/2} n)\cap B_2(r_n,P)} \biggl|\sum_{i=1}^n \xi_i f^2(X_i)\biggr|\nonumber\\
  &\lesssim r_n^2 + \frac{\log^{1/2} n}{n} \mathbb{E}\sup_{f\in\mathcal{F}_d\cap B_\infty(4\log^{1/2} n)\cap B_2(r_n,P)} \biggl|\sum_{i=1}^n \xi_i f(X_i)\biggr| \nonumber\\
  &\lesssim_{d,m_0,M_0} r_n^2 +  r_n n^{-1/d} \log^{\gamma_d} n \lesssim r_n^2,\label{Eq:SymmetrisedProcess}
\end{align}
where the second line uses the symmetrisation inequality \citep[cf.][Lemma~2.3.1]{vanderVaartWellner1996}, the third inequality follows from Lemma~\ref{Lemma:Contraction} and the penultimate inequality follows from applying Proposition~\ref{Prop:SymmetrisedProcess} to $f/(4\log^{1/2} n)$. For the second term on the right-hand side of~\eqref{Eq:Decomposition}, we first claim that there exists some constant $C'_{d,m_0,M_0} > 0$, depending only on $d, m_0$ and $M_0$, such that
\begin{equation}
  \label{Eq:RatioEmpiricalProcess}
  \mathbb{P}\biggl(\sup_{f\in\mathcal{F}_d\cap B_\infty(4\log^{1/2} n)\cap B_2(r_n,P)^c} \biggl|\frac{\mathbb{P}_n f^2}{P f^2} - 1\biggr| > C'_{d,m_0,M_0}\biggr)\leq \frac{2}{n^2}.    
\end{equation}
To see this, we adopt a peeling argument as follows. Let $\mathcal{F}_{d,\ell} := \{f\in\mathcal{F}_d\cap B_\infty(4\log^{1/2} n): 2^{\ell-1} r_n^2 < P f^2 \leq 2^\ell r_n^2\}$ and $m$ be the largest integer such that $2^mr_n^2< 32\log n$ (so that $m \asymp \log n$).  We have that
\begin{align*}
  \sup_{f\in\mathcal{F}_d\cap B_\infty(4\log^{1/2} n)\cap B_2(r_n,P)^c} \biggl|\frac{\mathbb{P}_n f^2}{P f^2} - 1\biggr| &= \frac{1}{n^{1/2}} \sup_{f\in\mathcal{F}_d\cap B_\infty(4\log^{1/2} n)\cap B_2(r_n,P)^c} \frac{|\mathbb{G}_n f^2|}{Pf^2} \\
  &\lesssim \frac{1}{n^{1/2}} \max_{\ell = 1,\ldots,m} \Bigl\{(2^\ell r_n^2)^{-1} \sup_{f\in\mathcal{F}_{d,\ell}} |\mathbb{G}_n f^2|\Bigr\}.
\end{align*}
By Talagrand's concentration inequality (cf. \cite{Talagrand1996}) for empirical processes, in the form given by \citet[Theorem~3]{Massart2000}, applied to the class $\{ f^2: f\in \mathcal{F}_{d,\ell}\}$, we have that for any $s_\ell > 0$,
\[
  \mathbb{P}\biggl\{\sup_{f\in\mathcal{F}_{d,\ell}} |\mathbb{G}_n f^2| > 2\mathbb{E}\sup_{f\in\mathcal{F}_{d,\ell}}|\mathbb{G}_n f^2| + 2^{(\ell+7)/2} r_n s_\ell^{1/2} \log^{1/2} n + \frac{552 s_\ell \log n}{n^{1/2}}\biggr\}\leq e^{-s_\ell}.
\]
Here we have used the fact that $\sup_{f\in\mathcal{F}_{d,\ell}}\mathrm{Var}_P f^2 \leq \sup_{f\in\mathcal{F}_{d,\ell}} P f^2 \|f\|_\infty^2 \leq 2^{\ell+4} r_n^2 \log n$. Further, we note by the symmetrisation inequality again, Lemma~\ref{Lemma:Contraction} and Proposition~\ref{Prop:SymmetrisedProcess} that
\begin{align*}
  \mathbb{E}\sup_{f\in\mathcal{F}_{d,\ell}} |\mathbb{G}_n f^2| &\lesssim \frac{1}{n^{1/2}}\mathbb{E}\sup_{f\in\mathcal{F}_{d,\ell}} \biggl|\sum_{i=1}^n \xi_i f^2(X_i)\biggr|\lesssim \frac{\log^{1/2} n}{n^{1/2}} \mathbb{E}\sup_{f\in\mathcal{F}_{d,\ell}}\biggl|\sum_{i=1}^n \xi_i f(X_i)\biggr| \\
  &\lesssim_{d,m_0,M_0} 2^{\ell/2} r_n n^{1/2-1/d} \log^{\gamma_d} n .
\end{align*}
By a union bound, we have that with probability at least $1 - \sum_{\ell=1}^m e^{-s_\ell}$,
\begin{align*}
  \sup_{f\in\mathcal{F}_d\cap B_\infty(4\log^{1/2} n)\cap B_2(r_n,P)^c} \biggl|&\frac{\mathbb{P}_nf^2}{Pf^2} - 1\biggr| \\
&\lesssim_{d,m_0,M_0} \max_{\ell = 1,\ldots,m}\biggl\{\frac{n^{1/2-1/d}\log^{\gamma_d} n + s_\ell^{1/2} \log^{1/2}n }{2^{\ell/2} n^{1/2} r_n} + \frac{s_\ell \log n}{2^\ell n r_n^2}\biggr\}.
\end{align*}
By choosing $s_\ell := 2^{\ell} \log n$, we see that $\sum_{\ell=1}^m e^{-s_\ell} \leq \sum_{\ell=1}^\infty n^{-\ell-1} \leq 2n^{-2}$ and 
\[
  \sup_{f\in\mathcal{F}_d\cap B_\infty(4\log^{1/2} n)\cap B_2(r_n,P)^c} \biggl|\frac{\mathbb{P}_nf^2}{Pf^2} - 1\biggr| \lesssim_{d,m_0,M_0} 1,
\]
which verifies~\eqref{Eq:RatioEmpiricalProcess}.  Now let $\mathcal{E} := \bigl\{\sup_{f\in\mathcal{F}_d\cap B_\infty(4\log^{1/2} n)\cap B_2(r_n,P)^c} \bigl|\frac{\mathbb{P}_n f^2}{P f^2}-1\bigr|\leq C'_{d,m_0,M_0}\bigr\}$.  Then
\begin{align}
  \mathbb{E}\bigl\{\bigl\|\tilde f_n\bigr\|_{L_2(\mathbb{P}_n)}^2 \mathbbm{1}_{\{\|\hat f_n\|_{L_2(P)} > r_n\}}\bigr\} &\leq \mathbb{E}\bigl\{\bigl\|\tilde f_n\bigr\|_{L_2(\mathbb{P}_n)}^2 \mathbbm{1}_{\{\|\hat f_n\|_{L_2(P)} > r_n\}}\mathbbm{1}_\mathcal{E}\bigr\} + \frac{32\log n}{n^2}\nonumber\\
  & \leq (C'_{d,m_0,M_0}+1)\mathbb{E}\bigl\|\tilde f_n\bigr\|_{L_2(P)}^2 + \frac{32\log n}{n^2}.\label{Eq:secondterm}
\end{align}
Combining~\eqref{Eq:Decomposition},~\eqref{Eq:SymmetrisedProcess} and~\eqref{Eq:secondterm}, we obtain
\[
  \mathbb{E}\bigl\|\tilde f_n\bigr\|_{L_2(\mathbb{P}_n)}^2 \lesssim_{d,m_0,M_0} r_n^2 + \mathbb{E}\bigl\|\tilde f_n\bigr\|_{L_2(P)}^2,
\]
as desired. 
\end{proof}

\begin{proof}[Proof of Proposition~\ref{Prop:RandomAdaptUpper}]
For $f_0 = 0$, we decompose
\begin{equation}
  \label{Eq:TwoTerms}
  R(\hat{f}_n,0) = \mathbb{E}\bigl\{\bigl\|\hat f_n\bigr\|_{L_2(\mathbb{P}_n)}^2 \mathbbm{1}_{\{\|\hat f_n\|_\infty\leq 4\log^{1/2} n\}}\bigr\} + \mathbb{E}\bigl\{\bigl\|\hat f_n\bigr\|_{L_2(\mathbb{P}_n)}^2 \mathbbm{1}_{\{\|\hat f_n\|_\infty> 4\log^{1/2} n\}}\bigr\}
\end{equation}
and handle the two terms on the right-hand side separately. For the first term, let $r_n := n^{-1/d}\log^{\gamma_d} n$ and observe that by Lemma~\ref{Lemma:Conversion} and Proposition~\ref{Prop:SymmetrisedProcess}, we have that for $r \geq r_n$,
\[
  \mathbb{E}\sup_{f\in\mathcal{F}_d\cap B_\infty(4\log^{1/2} n)\cap B_2(r,P)}\biggl|\frac{1}{n^{1/2}}\sum_{i=1}^n \epsilon_i f(X_i)\biggr| \lesssim_{d,m_0,M_0} r n^{1/2-1/d} \log^{\gamma_d} n.
\]
On the other hand, by Lemma~\ref{Lemma:Contraction} and Proposition~\ref{Prop:SymmetrisedProcess}, for $r \geq r_n$,
\[
  \mathbb{E}\sup_{f\in\mathcal{F}_d\cap B_\infty(4\log^{1/2} n)\cap B_2(r,P)}\biggl|\frac{1}{n^{1/2}}\sum_{i=1}^n \xi_i f^2(X_i)\biggr| \lesssim_{d,m_0,M_0} r n^{1/2-1/d} \log^{\gamma_d} n.
\]
It follows that the conditions of Proposition~\ref{Prop:Reduction} are satisfied for this choice of $r_n$ with $\phi_n(r) := r n^{1/2-1/d} \log^{\gamma_d} n$.  By Propositions~\ref{Prop:RatioEmpiricalProcess} and~\ref{Prop:Reduction}, we deduce that
\begin{align}
\label{Eq:FirstTerm}
  \mathbb{E} \bigl\{\bigl\|\hat f_n\bigr\|_{L_2(\mathbb{P}_n)}^2  \mathbbm{1}_{\{\|\hat f_n\|_\infty\leq 4\log^{1/2} n\}}\bigr\} &\lesssim_{d,m_0,M_0} n^{-2/d}\log^{2\gamma_d}n + \mathbb{E} \bigl\{\bigl\|\hat f_n\bigr\|_{L_2(P)}^2  \mathbbm{1}_{\{\|\hat f_n\|_\infty\leq 4\log^{1/2} n\}}\bigr\} \nonumber \\
&\lesssim_{d,m_0,M_0} n^{-2/d}\log^{2\gamma_d} n.
\end{align}
For the second term on the right-hand side of~\eqref{Eq:TwoTerms}, we note that by the definition of the least squares estimator, $\sum_{i=1}^n \{\hat f_n(X_i) - \epsilon_i\}^2 \leq \sum_{i=1}^n \epsilon_i^2$, so
\[
  \bigl\|\hat f_n\bigr\|_{L_2(\mathbb{P}_n)}^2 \leq \frac{2}{n}\sum_{i=1}^n \{\hat f_n(X_i) - \epsilon_i\}^2 + \frac{2}{n}\sum_{i=1}^n \epsilon_i^2 \leq \frac{4}{n}\sum_{i=1}^n \epsilon_i^2.
\]
Thus, 
\begin{align}
  \mathbb{E} \bigl\{\bigl\|\hat f_n\bigr\|_{L_2(\mathbb{P}_n)}^2  \mathbbm{1}_{\{\|\hat f_n\|_\infty> 4\log^{1/2} n\}}\bigr\} &\leq 4\mathbb{E} \bigl\{\epsilon_1^2  \mathbbm{1}_{\{\|\hat f_n\|_\infty> 4\log^{1/2} n\}}\bigr\}\nonumber\\
  &\lesssim \mathbb{P}(\|\hat f_n\|_\infty > 4\log^{1/2} n)^{1/2} \lesssim  n^{-3},\label{Eq:SecondTerm}
\end{align}
where the final inequality follows from Lemma~\ref{Lemma:Linfty}. The proof is completed by substituting~\eqref{Eq:FirstTerm} and~\eqref{Eq:SecondTerm} into~\eqref{Eq:TwoTerms}.
\end{proof}

\section{Appendix: proofs of ancillary results}
\label{Sec:Appendix}
The proof of Corollary~\ref{Cor:Fixed} requires the following lemma on Riemann approximation of block increasing functions. 

\begin{lemma}
\label{Lemma:RiemannApproximation}
Suppose $n_1 = n^{1/d}$ is a positive integer. For any $f \in \mathcal{F}_d$, define $f_L(x_1,\ldots,x_d) := f\bigl(n_1^{-1}\lfloor n_1 x_1\rfloor, \ldots, n_1^{-1}\lfloor n_1 x_d\rfloor \bigr)$ and $f_U(x_1,\ldots,x_d) := f\bigl(n_1^{-1}\lceil n_1 x_1\rceil, \ldots, n_1^{-1}\lceil n_1 x_d\rceil \bigr)$. Then
 \[
  \int_{[0,1]^d} (f_U - f_L)^2 \leq 4d n^{-1/d}\|f\|_{\infty}^2.
 \]
\end{lemma}

\begin{proof}
For $x = (x_1,\ldots,x_d)^\top$ and $x' = (x'_1,\ldots,x'_d)^\top$ in $\mathbb{L}_{d,n}$, we say $x$ and $x'$ are equivalent if and only if $x_j - x_1 = x'_j - x'_1$ for $j = 1,\ldots,d$. Let $\mathbb{L}_{d,n} = \bigsqcup_{r=1}^N P_r$ be the partition of $\mathbb{L}_{d,n}$ into equivalence classes. Since each $P_r$ has non-empty intersection with a different element of the set $\{(x_1,\ldots,x_d)\in\mathbb{L}_{d,n}: \min_j x_j = 1\}$, we must have $N\leq d n^{1-1/d}$. Therefore, we have
 \begin{align*}
  \int_{[0,1]^d} (f_U - f_L)^2 & = \sum_{r=1}^N \int_{n_1^{-1}(P_r + (-1,0]^d)} (f_U - f_L)^2 \\
  & \leq \frac{2}{n}\|f\|_\infty \sum_{r=1}^N \sum_{x = (x_1,\ldots,x_d)^{\top} \in P_r} \biggl\{f\biggl(\frac{x_1}{n},\ldots,\frac{x_d}{n}\biggr) - f\biggl(\frac{x_1-1}{n},\ldots,\frac{x_d-1}{n}\biggr)\biggr\} \\
  & \leq \frac{2N}{n}\|f\|_\infty \bigl(f(1,\ldots,1) - f(0,\ldots,0)\bigr) \leq 4d n^{-1/d}\|f\|_\infty^2,
 \end{align*} 
 as desired.
\end{proof}

The following is a simple generalisation of Jensen's inequality. 
\begin{lemma}
  \label{Lemma:Jensen+}
  Suppose $h: [0,\infty) \to (0,\infty)$ is a non-decreasing function satisfying the following:
  \begin{enumerate}[label={\textup{(\roman*)}}, noitemsep]
    \item There exists $x_0\geq 0$ such that $h$ is concave on $[x_0,\infty)$.
    \item There exists some $x_1>x_0$ such that $h(x_1)-x h'_+(x_1) \geq h(x_0)$, where $h'_+$ is the right derivative of $h$.
  \end{enumerate}
  Then there exists a constant $C_h > 0$ depending only on $h$ such that for any nonnegative random variable $X$ with $\mathbb{E} X <\infty$, we have 
  \[
    \mathbb{E}h(X) \leq C_h h(\mathbb{E}X).
  \]
\end{lemma}
\begin{proof}
Define $H:[0,\infty) \rightarrow [h(0),\infty)$ by
\[
H(x) := \begin{cases} h(x_1) - x_1h_+'(x_1) + xh_+'(x_1) & \text{if $x \in [0,x_1)$} \\
h(x) & \text{if $x \in [x_1,\infty)$.} \end{cases}
\]
Then $H$ is a concave majorant of $h$.  Moreover, we have $H \leq (h(x_1)/h(0)) h$. Hence, by Jensen's inequality, we have
  \[
    \mathbb{E} h(X) \leq \mathbb{E} H(X) \leq H(\mathbb{E}X) \leq \frac{h(x_1)}{h(0)} h(\mathbb{E}X),
  \]
  as desired.
\end{proof}

We need the following lower bound on the metric entropy of $\mathcal{M}(\mathbb{L}_{2,n})\cap B_2(1)$ for the proof of Proposition~\ref{Prop:FixedAdaptLower}.
\begin{lemma}
\label{Lemma:MetricEntropyLowerBound}
 There exist universal constants $c > 0$ and $\varepsilon_0 > 0$ such that
 \[
  \log N\bigl(\varepsilon_0, \mathcal{M}(\mathbb{L}_{2,n})\cap B_2(1), \|\cdot\|_2\bigr) \geq c \log^2 n.
 \]
\end{lemma}

\begin{proof}
It suffices to prove the equivalent result that there exist universal constants $c, \varepsilon_0 > 0$ such that the packing number $D\bigl(\varepsilon_0,\mathcal{M}(\mathbb{L}_{2,n})\cap B_2(1), \|\cdot\|_2\bigr)$ (i.e.\ the maximum number of disjoint open Euclidean balls of radius $\varepsilon_0$ that can be fitted into $\mathcal{M}(\mathbb{L}_{2,n})\cap B_2(1)$) is at least $\exp(c\log^2 n)$. Without loss of generality, we may also assume that $n_1: = n^{1/2} = 2^{\ell}-1$ for some $\ell \in \mathbb{N}$, so that $\ell \asymp \log n$. Now, for $r=1,\ldots,\ell$, let $I_r := \{2^{r-1},\ldots,2^r-1\}$ and consider the set
 \[
  \bar{\mathcal{M}} := \biggl\{\theta\in\mathbb{R}^{\mathbb{L}_{2,n}} : \theta_{I_r\times I_s} \in \Bigl\{\frac{-\mathbf{1}_{I_r\times I_s}}{\sqrt{2^{r+s+1}}\log n}, \frac{-\mathbf{1}_{I_r\times I_s}}{\sqrt{2^{r+s}}\log n} \Bigr\}\biggr\} \subseteq \mathcal{M}(\mathbb{L}_{2,n})\cap B_2(1),
 \]
 where $\mathbf{1}_{I_r\times I_s}$ denotes the all-one vector on $I_r\times I_s$. Define a bijection $\psi: \bar{\mathcal{M}} \to \{0,1\}^{\ell^2}$ by
 \[
  \psi(\theta) := \Bigl( \mathbbm{1}_{\bigl\{\theta_{I_r\times I_s}= - \mathbf{1}_{I_r\times I_s}/\sqrt{2^{r+s+1}}\log n\bigr\}}\Bigr)_{r,s=1}^\ell.
 \]
Then, for $\theta, \theta' \in \bar{\mathcal{M}}$,
 \[
  \|\theta-\theta'\|_2^2 = \frac{d_{\mathrm{H}}(\psi(\theta), \psi(\theta'))}{\log^2 n}\frac{1}{4}\biggl(1 - \frac{1}{2^{1/2}}\biggr)^2,
 \]
where $d_{\mathrm{H}}(\cdot,\cdot)$ denotes the Hamming distance. On the other hand, the Gilbert--Varshamov bound \citep[e.g.][Lemma~4.7]{Massart2007} entails that there exists a subset $\mathcal{I} \subseteq \{0,1\}^{\ell^2}$ such that $|\mathcal{I}| \geq \exp(\ell^2/8)$ and $d_{\mathrm{H}}(v,v')\geq \ell^2/4$ for any distinct $v, v'\in\mathcal{I}$. Then the set $\psi^{-1}(\mathcal{I}) \subseteq \bar{\mathcal{M}}$ has cardinality at least $\exp(\ell^2/8) \geq \exp(\log^2 n/32)$, and each pair of distinct elements have squared $\ell_2$ distance at least $\varepsilon_0 := \frac{\ell^2/4}{\log^2 n}\frac{1}{4}(1 - \frac{1}{2^{1/2}})^2 \gtrsim 1$, as desired.
\end{proof}

Lemma~\ref{Lemma:RandomAntichain} below gives a lower bound on the size of the maximal antichain (with respect to the natural partial ordering on $\mathbb{R}^d$) among independent and identically distributed $X_1,\ldots,X_n$.
\begin{lemma}
\label{Lemma:RandomAntichain}
  Let $d\geq 2$.  Let $X_1,\ldots,X_n \stackrel{\mathrm{iid}}{\sim}  P$, where $P$ is a distribution on $[0,1]^d$ with Lebesgue density bounded above by $M_0\in [1,\infty)$.  Then with probability at least $1- e^{-ed^{-1} (M_0 n)^{1/d}\log (M_0 n)}$, there is an antichain in $G_X$ with cardinality at least $n^{1-1/d}/(2eM_0^{1/d})$.
\end{lemma}
\begin{proof}
  By Dilworth's Theorem \citep{Dilworth1950}, for each realisation of the directed acyclic graph $G_X$, there exists a covering of $V(G_X)$ by chains $\mathcal{C}_1,\ldots,\mathcal{C}_M$, where $M$ denotes the cardinality of a maximum antichain of $G_X$. Thus, it suffices to show that with the given probability, the maximum chain length of $G_X$ is at most $k:= \lceil e (M_0 n)^{1/d}\rceil \leq 2e(M_0 n)^{1/d}$. By a union bound, we have that
  \begin{align*}
    \mathbb{P}(\text{$\exists$ a chain of length $k$ in $G_X$}) &\leq \binom{n}{k}\mathbb{P}(X_1\preceq \cdots \preceq X_k) = \binom{n}{k}(k!)^{-d}M_0^k\\
    & \leq \biggl(\frac{en}{k}\biggr)^k\biggl(\frac{k}{e}\biggr)^{-kd} M_0^k \leq (M_0 n)^{-k/d} \leq e^{-ed^{-1} (M_0 n)^{1/d}\log (M_0 n)},
  \end{align*}
  as desired.
\end{proof}

The following two lemmas control the empirical processes in~\eqref{Eq:GaussianMultiplier} and~\eqref{Eq:RademacherQuadratic} by the symmetrised empirical process in~\eqref{Eq:RademacherMultiplier}.

\begin{lemma}
 \label{Lemma:Conversion}
Let $n \geq 2$, and suppose $X_1,\ldots,X_n$ are independent and identically distributed on $\mathcal{X}$. Then for any countable class $\mathcal{F}$ of measurable, real-valued functions defined on $\mathcal{X}$, we have
 \[
   \mathbb{E}\sup_{f\in\mathcal{F}}\biggl|\sum_{i=1}^n \epsilon_i f(X_i)\biggr| \leq 2 \log^{1/2} n \, \mathbb{E}\sup_{f\in\mathcal{F}}\biggl|\sum_{i=1}^n \xi_i f(X_i)\biggr|.
 \]
\end{lemma}

\begin{proof}
Let $\alpha_0 := 0$, and for $k=1,\ldots,n$, let $\alpha_k := \mathbb{E} |\epsilon_{(k)}|$, where $|\epsilon_{(1)}|\leq \cdots\leq |\epsilon_{(n)}|$ are the order statistics of $\{|\epsilon_1|,\ldots,|\epsilon_n|\}$, so that $\alpha_n \leq (2\log n)^{1/2}$.  Observe that for any $k=1,\ldots,n$,
\begin{align}
\label{Eq:Obvious?}
\mathbb{E}\sup_{f\in\mathcal{F}}\biggl|\sum_{i=1}^k\xi_i f(X_i)\biggr| &= \mathbb{E}\sup_{f\in\mathcal{F}}\biggl|\sum_{i=1}^k\xi_i f(X_i) + \mathbb{E}\sum_{i=k+1}^n\xi_i f(X_i)\biggr| \nonumber \\
&\leq \mathbb{E}\sup_{f\in\mathcal{F}} \mathbb{E}\biggl\{\biggl|\sum_{i=1}^n\xi_i f(X_i)\biggr| \biggm| X_1,\ldots,X_k,\xi_1,\ldots,\xi_k\biggr\} \leq \mathbb{E} \sup_{f\in\mathcal{F}}\biggl|\sum_{i=1}^n\xi_i f(X_i)\biggr|.
\end{align}
We deduce from \citet[Proposition~1]{HanWellner2017} and~\eqref{Eq:Obvious?} that
  \[
    \mathbb{E}\sup_{f\in\mathcal{F}}\biggl|\sum_{i=1}^n\epsilon_i f(X_i)\biggr| \leq 2^{1/2}\sum_{k=1}^n (\alpha_k - \alpha_{k-1}) \mathbb{E}\sup_{f\in\mathcal{F}}\biggl|\sum_{i=1}^k\xi_i f(X_i)\biggr| \leq 2^{1/2}\alpha_n\mathbb{E} \sup_{f\in\mathcal{F}}\biggl|\sum_{i=1}^n\xi_i f(X_i)\biggr|,
  \]
as required.
\end{proof}

\begin{lemma}
	\label{Lemma:Contraction}
	Let $X_1,\ldots,X_n$ be random variables taking values in $\mathcal{X}$ and $\mathcal{F}$ be a countable class of measurable functions $f:\mathcal{X} \rightarrow [-1,1]$. Then
	\[
	\mathbb{E}\sup_{f\in\mathcal{F}}\biggl|\sum_{i=1}^n \xi_i f^2(X_i)\biggr| \leq 4\mathbb{E}\sup_{f\in\mathcal{F}}\biggl|\sum_{i=1}^n \xi_i f(X_i)\biggr|.
	\]
\end{lemma}

\begin{proof}
	By \citet[Theorem~4.12]{LedouxTalagrand2011}, applied to $\phi_i(y) = y^2/2$ for $i = 1,\ldots,n$ (note that $y\mapsto y^2/2$ is a contraction on $[0,1]$), we have
	\begin{align*}
	\mathbb{E}\sup_{f\in\mathcal{F}}\biggl|\sum_{i=1}^n \xi_i f^2(X_i)\biggr| &= \mathbb{E}\biggl\{\mathbb{E}\sup_{f\in\mathcal{F}}\Bigl|\sum_{i=1}^n \xi_i f^2(X_i)\Bigr| \biggm| X_1,\ldots,X_n\biggr\} \\
	&\leq 4\mathbb{E}\biggl\{\mathbb{E}\sup_{f\in\mathcal{F}}\Bigl|\sum_{i=1}^n \xi_i f(X_i)\Bigr| \biggm| X_1,\ldots,X_n\biggr\} = 4\mathbb{E}\sup_{f\in\mathcal{F}}\biggl|\sum_{i=1}^n \xi_i f(X_i)\biggr|,
	\end{align*}
	as required.
\end{proof}

The following is a local maximal inequality for empirical processes under bracketing entropy conditions. This result is well known for $\eta = 0$ in the literature, but we provide a proof for the general case $\eta \geq 0$ for the convenience of the reader. 

\begin{lemma}
 \label{Lemma:LocalMaximalInequality}
Let $X_1,\ldots,X_n \stackrel{\mathrm{iid}}{\sim} P$ on $\mathcal{X}$ with empirical distribution $\mathbb{P}_n$, and, for some $r > 0$, let $\mathcal{G} \subseteq B_\infty(1)\cap B_2(r,P)$ be a countable class of measurable functions. Then for any $\eta \in [0,r/3)$, we have
 \[
  \mathbb{E} \sup_{f\in\mathcal{G}} |\mathbb{G}_n f| \lesssim n^{1/2}\eta + \int_{\eta}^r  \log_+^{1/2} N_{[\,]}(\varepsilon, \mathcal{G}, \|\cdot\|_{L_2(P)})\,\mathrm{d}\varepsilon + n^{-1/2} \log_+ N_{[\,]}(r, \mathcal{G}, \|\cdot\|_{L_2(P)}).
 \] 
 The above inequality also holds if we replace $\mathbb{G}_n f$ with the symmetrised empirical process $n^{-1/2}\sum_{i=1}^n \xi_i f(X_i)$. 
\end{lemma}

\begin{proof}
Writing $N_r := N_{[\,]}(r,\mathcal{G},\|\cdot\|_{L_2(P)})$, there exists $\{(f_\ell^L,f_\ell^U):\ell=1,\ldots,N_r\}$ that form an $r$-bracketing set for $\mathcal{G}$ in the $L_2(P)$ norm.  Letting $\mathcal{G}_1 := \{f \in \mathcal{G}:f_1^L \leq f \leq f_1^U\}$ and $\mathcal{G}_\ell := \{f \in \mathcal{G}:f_\ell^L \leq f \leq f_\ell^U\} \setminus \cup_{j=1}^{\ell-1} \mathcal{G}_j$ for $\ell=2,\ldots,N_r$, we see that $\{\mathcal{G}_\ell\}_{\ell=1}^{N_r}$ is a partition of $\mathcal{G}$ such that the $L_2(P)$-diameter of each $\mathcal{G}_\ell$ is at most $r$.  It follows by~\citet[Lemma~2.14.3]{vanderVaartWellner1996} that for any choice of $f_\ell \in \mathcal{G}_\ell$, we have that
 \begin{align}
   \mathbb{E} \sup_{f\in\mathcal{G}} |\mathbb{G}_n f| \lesssim n^{1/2}\eta &+ \int_{\eta}^r \log_+^{1/2} N_{[\,]}(\varepsilon,\mathcal{G},\|\cdot\|_{L_2(P)})\,\mathrm{d}\varepsilon\nonumber\\
   & + \mathbb{E}\max_{\ell=1,\ldots,N_r} |\mathbb{G}_n f_\ell| + \mathbb{E}\max_{\ell=1,\ldots,N_r} \Bigl|\mathbb{G}_n\Bigl(\sup_{f\in\mathcal{G}_\ell} |f-f_\ell|\Bigr)\Bigr|,\label{Eq:Ossiander}
 \end{align}
The third and fourth terms of~\eqref{Eq:Ossiander} can be controlled by Bernstein's inequality (in the form of (2.5.5) in \citet{vanderVaartWellner1996}):
 \[
   \mathbb{E}\max_{\ell=1,\ldots,N_r} |\mathbb{G}_n f_\ell| \vee \mathbb{E}\max_{\ell=1,\ldots,N_r}\Bigl|\mathbb{G}_n\Bigl(\sup_{f\in\mathcal{G}_\ell} |f-f_\ell|\Bigr)\Bigr| \lesssim \frac{\log_+ N_r}{n^{1/2}} + r\log_+^{1/2} N_r.
 \]
 Since $\eta < r/3$, the last term $r\log_+^{1/2} N_r$ in the above display can be assimilated into the entropy integral in~\eqref{Eq:Ossiander}, which establishes the claim for $\mathbb{E}\sup_{f\in\mathcal{G}}|\mathbb{G}_n f|$. 
 
We now study the symmetrised empirical process.  For $f \in \mathcal{G}$, we define $e\otimes f:\{-1,1\} \times \mathcal{X} \rightarrow \mathbb{R}$ by $(e \otimes f)(t,x) := tf(x)$, and apply the previous result to the function class $e \otimes \mathcal{G} := \{e\otimes f: f\in\mathcal{G}\}$.  Here the randomness is induced by the independently and identically distributed pairs $(\xi_i, X_i)_{1\leq i\leq n}$.  For any $f\in\mathcal{G}$ and any $\varepsilon$-bracket $[\underline{f}, \bar f]$ containing $f$, we have that $[e_+\otimes \underline{f} - e_-\otimes \bar f , e_+\otimes \bar f - e_-\otimes \underline{f}]$ is an $\varepsilon$-bracket for $e\otimes f$ in the $L_2(P_\xi\otimes P)$ metric, where $e_+(t) := \max\{e(t),0\} = \max(t,0)$ and $e_-(t) = \max(-t,0)$.  Writing $P_\xi$ denote the Rademacher distribution on $\{-1,1\}$, it follows that for every $\epsilon > 0$,
 \[
   N_{[\,]}(\varepsilon, e\otimes \mathcal{G}, L_2(P_\xi \otimes P)) \leq N_{[\,]}(\varepsilon, \mathcal{G}, L_2(P)),
 \]
 which proves the claim for the symmetrised empirical process.
\end{proof}

In the next two lemmas, we assume, as in the main text, that $P$ is a distribution on $[0,1]^d$ with Lebesgue density bounded above and below by $M_0\in[1,\infty)$ and $m_0\in(0,1]$ respectively.  Recall that $\mathcal{F}_{d,\downarrow}^+ = \{f: -f\in\mathcal{F}_d, f\geq 0\}$.  The following result is used to control the bracketing entropy terms that appear in Lemma~\ref{Lemma:LocalMaximalInequality} when we apply it in the proof of Proposition~\ref{Prop:SymmetrisedProcess}. 
\begin{lemma}
 \label{Lemma:BracketingEntropyBound}
 There exists a constant $C_{d} > 0$, depending only on $d$, such that for any $r, \epsilon > 0$,
 \begin{align*}
  \log N_{[\,]}\bigl(\varepsilon, \mathcal{F}_{d,\downarrow}^+ \cap B_2(r, P) \cap &B_\infty(1), \|\cdot\|_{L_2(P)}\bigr) \\
&\leq 
  C_{d} \begin{cases}
              (r/\varepsilon)^2\frac{M_0}{m_0} \log^2(\frac{M_0}{m_0}) \log_+^4(1/\varepsilon) \log_+^2\bigl(\frac{r\log_+(1/\varepsilon)}{\varepsilon}\bigr) & \text{ if $d = 2$,}\\
              (r/\varepsilon)^{2(d-1)}(\frac{M_0}{m_0})^{d-1}\log_+^{d^2}(1/\varepsilon) & \text{ if $d\geq 3$.}
            \end{cases}
 \end{align*}
\end{lemma}

\begin{proof}
 For any Borel measurable $S\subseteq [0,1]^d$, we define $\|f\|_{L_2(P; S)}^2 := \int_S f^2\,\mathrm{d}P$. We first claim that for any $\eta \in (0,1/4]$, 
 \[
  \log N_{[\,]}\bigl(\varepsilon, \mathcal{F}_{d,\downarrow}^+\cap B_2(r), \|\cdot\|_{L_2(P;[\eta,1]^d)} \bigr) \lesssim_{d} 
  \begin{cases}
  (\frac{r}{\varepsilon})^{2} \frac{M_0}{m_0} \log^2(\frac{M_0}{m_0}) \log^{4}(1/\eta) \log_+^2\bigl(\frac{r\log(1/\eta)}{\varepsilon}\bigr)	  & \text{if $d=2$},\\
  (\frac{r}{\varepsilon})^{2(d-1)}(\frac{M_0}{m_0})^{d-1} \log^{d^2}(1/\eta)	  & \text{if $d\geq 3$}.
  \end{cases}
 \]
 By the cone property of $\mathcal{F}_{d,\downarrow}^+$, it suffices to establish the above claim when $r = 1$. We denote by $\mathrm{vol}(S)$ the Lebesgue measure of a measurable set $S\subseteq[0,1]^d$. By \citet[Theorem~1.1]{GaoWellner2007} and a scaling argument, we have for any $\delta, M > 0$ and any hyperrectangle $A\subseteq [0,1]^d$ that
 \begin{equation}
   \label{Eq:HyperrectangleEntropy}
    \log N_{[\,]}\bigl(\delta, \mathcal{F}_{d,\downarrow}^+\cap B_\infty(M), \|\cdot\|_{L_2(P;A)}\bigr) \lesssim_d
   \begin{cases}
     (\gamma/\delta)^2\log_+^2(\gamma/\delta) & \text{ if $d = 2$,}\\
     (\gamma/\delta)^{2(d-1)} & \text{ if $d\geq 3$},
   \end{cases}
 \end{equation}
where $\gamma := M_0^{1/2} M\mathrm{vol}^{1/2}(A)$.  Let $m := \lfloor \log_2(1/\eta)\rfloor$ and define $I_\ell := [2^\ell\eta, 2^{\ell+1}\eta] \cap [0,1]$ for each $\ell = 0,\ldots,m$. Then for $\ell_1,\ldots,\ell_d\in\{0,\ldots,m\}$, any $f\in\mathcal{F}_{d,\downarrow}^+\cap B_2(1,P)$ is uniformly bounded by $\bigl(m_0\prod_{j=1}^d 2^{\ell_j} \eta\bigr)^{-1/2}$ on the hyperrectangle $\prod_{j=1}^d I_{\ell_j}$. Then by~\eqref{Eq:HyperrectangleEntropy} we see that for any $\delta > 0$, 
 \[
   \log N_{[\,]}\bigl(\delta, \mathcal{F}_{d,\downarrow}^+\cap B_2(1), \|\cdot\|_{L_2(P;\prod_{j=1}^d I_{\ell_j})}\bigr) \lesssim_d
   \begin{cases}
     \delta^{-2}(M_0/m_0)\log^2(\frac{M_0}{m_0}) \log^2_+(1/\delta) & \text{ if $d = 2$,}\\
     \delta^{-2(d-1)}(M_0/m_0)^{d-1} & \text{ if $d \geq 3$},
   \end{cases}
 \]
where we have used the fact that $\log_+(ax) \leq 2\log_+(a)\log_+(x)$ for any $a,x > 0$. Global brackets for $\mathcal{F}_{d,\downarrow}^+\cap B_2(1)$ on $[\eta,1]^d$ can then be constructed by taking all possible combinations of local brackets on $I_{\ell_1}\times\cdots\times I_{\ell_d}$ for $\ell_1,\ldots,\ell_d\in\{0,\ldots,m\}$. Overall, for any $\varepsilon > 0$, setting $\delta = (m+1)^{-d/2}\varepsilon$ establishes the claim.  We conclude that if we fix any $\varepsilon > 0$, take $\eta = \varepsilon^2/(4d) \wedge 1/4$ and take a single bracket consisting of the constant functions 0 and 1 on $[0,1]^d \setminus [\eta,1]^d$, we have 
\begin{align*}
  \log N_{[\,]}\bigl(\varepsilon, \mathcal{F}_{d,\downarrow}^+\cap B_2(r)\cap B_\infty(1), &\|\cdot\|_{L_2(P)}\bigr) \leq \log N_{[\,]}\bigl(\varepsilon/2, \mathcal{F}_{d,\downarrow}^+\cap B_2(r), \|\cdot\|_{L_2(P;[\eta,1]^d)}\bigr)\\
  \lesssim_{d} & \begin{cases}
              (r/\varepsilon)^2\frac{M_0}{m_0}\log^2(\frac{M_0}{m_0}) \log_+^4(1/\varepsilon) \log_+^2\bigl(\frac{r\log_+(1/\varepsilon)}{\varepsilon}\bigr) & \text{ if $d = 2$,}\\
              (r/\varepsilon)^{2(d-1)}(\frac{M_0}{m_0})^{d-1}\log_+^{d^2}(1/\varepsilon) & \text{ if $d\geq 3$,}
            \end{cases}
\end{align*}
completing the proof.
\end{proof}

For $0<r<1$, let $F_r$ be the envelope function of $\mathcal{F}^+_{d,\downarrow}\cap B_2(r, P) \cap B_\infty(1)$.  The lemma below controls the $L_2(P)$ norm of $F_r$ when restricted to strips of the form $I_\ell:= [0,1]^{d-1} \times [\frac{\ell-1}{n_1}, \frac{\ell}{n_1}]$ for $\ell = 1,\ldots,n_1$. 
\begin{lemma}
\label{Lemma:EnvelopeIntegral}
For any $r \in (0,1]$ and $\ell =1,\ldots,n_1$, we have
  \[
    \int_{I_\ell} F_r^2\,\mathrm{d}P \leq \frac{7M_0 r^2 \log_+^d(1/r^2)}{m_0\ell}.
  \]
\end{lemma}

\begin{proof}
By monotonicity and the $L_2(P)$ and $L_\infty$ constraints, we have $F_r^2(x_1,\ldots,x_d) \leq \frac{r^2}{m_0 x_1\cdots x_d} \wedge 1$.  We first claim that for any $d\in\mathbb{N}$,
  \[
    \int_{[0,1]^d} \biggl(\frac{t}{x_1\cdots x_d}\wedge 1\biggr) \, \mathrm{d}x_1\cdots \mathrm{d}x_d \leq 5t\log_+^d(1/t).
  \]
  To see this, we define $S_d := \bigl\{(x_1,\ldots,x_d): \prod_{j=1}^d x_j \geq t\bigr\}$ and set $a_d := \int_{S_d} \frac{t}{x_1\cdots x_d}\, \mathrm{d}x_1\cdots\mathrm{d}x_d$ and $b_d := \int_{S_d} \mathrm{d}x_1\cdots\mathrm{d}x_d$. By integrating out the last coordinate, we obtain the following relation
  \begin{equation}
  \label{Eq:adbd}
    b_d = \int_{S_{d-1}}\biggl(1-\frac{t}{x_1\cdots x_{d-1}}\biggr)\, \mathrm{d}x_1\cdots \mathrm{d}x_{d-1} = b_{d-1}-a_{d-1}.
  \end{equation}
  On the other hand, we have by direct computation that
  \begin{align}
    a_d &= \int_{t}^1 \cdots \int_{\frac{t}{x_1\cdots x_{d-1}}}^1 \frac{t}{x_1\cdots x_d} \, \mathrm{d}x_d\cdots \mathrm{d}x_1 \nonumber\\
    &\leq a_{d-1}\log(1/t) \leq \cdots \leq a_1\log^{d-1}(1/t) = t\log^{d}(1/t).\label{Eq:ad}
  \end{align}
  Combining~\eqref{Eq:adbd} and~\eqref{Eq:ad}, we have
  \begin{align*}
     \int_{[0,1]^d} \biggl(\frac{t}{x_1\cdots x_d}\wedge 1\biggr) \, \mathrm{d}x_1\cdots \mathrm{d}x_{d} &= a_d + 1 - b_d \leq \min\{a_d + 1, a_d + a_{d-1} + \cdots + a_1 + 1 - b_1\}\\
     &\leq \min\biggl\{t\log^d(1/t) + 1, \frac{t\log^{d+1}(1/t)}{\log(1/t) - 1}\biggr\} \leq 5t\log_+^d(1/t),
  \end{align*}
  as claimed, where the final inequality follows by considering the cases $t \in [1/e,1]$, $t \in [1/4,1/e)$ and $t \in [0,1/4)$ separately. Consequently, for $\ell = 2,\ldots,n_1$, we have that
  \begin{align*}
    \int_{I_\ell} F_r^2\,\mathrm{d}P &\leq \frac{M_0}{m_0}\int_{(\ell-1)/n_1}^{\ell/n_1} \int_{[0,1]^{d-1}} \biggl(\frac{r^2/x_d}{x_1\cdots x_{d-1}}\wedge 1\biggr)\, \mathrm{d}x_1\cdots \mathrm{d}x_{d-1}\mathrm{d}x_d \\
    &\leq \frac{M_0}{m_0}\int_{(\ell-1)/n_1}^{\ell/n_1} 5(r^2/x_d)\log_+^{d-1}(x_d/r^2)\, \mathrm{d}x_d \\
    &\leq \frac{M_0}{m_0}5r^2 \log_+^{d-1} (1/r^2) \log\bigl(\ell/(\ell-1)\bigr) \leq \frac{7M_0r^2 \log_+^{d-1}(1/r^2)}{m_0\ell},
  \end{align*}
  as desired. For the remaining case $\ell = 1$, we have 
  \[
    \int_{I_1} F_r^2\,\mathrm{d}P \leq M_0\int_{[0,1]^d} F_r^2 \,\mathrm{d}x_1\cdots\mathrm{d}x_d \leq \frac{5M_0}{m_0}r^2 \log_+^d(1/r^2),
  \]
  which is also of the correct form.
\end{proof}

The following is a simple tail bound for $\|\hat{f}_n\|_\infty$.
\begin{lemma}
  \label{Lemma:Linfty}
  For $f_0 = 0$, we have
  \[
    \mathbb{P}\bigl(\|\hat f_n\|_\infty \geq 4\log^{1/2} n\bigr)\leq 2n^{-7}.
  \]
\end{lemma}

\begin{proof}
Recall that we say $U \subseteq \mathbb{R}^d$ is an \emph{upper set} if whenever $x \in U$ and $x \preceq y$, we have $y \in U$; we say, $L \subseteq \mathbb{R}^d$ is a \emph{lower set} if $-L$ is an upper set.   We write $\mathcal{U}$ and $\mathcal{L}$ respectively for the collections of upper and lower sets in $\mathbb{R}^d$.  The least squares estimator $\hat f_n$ over $\mathcal{F}_d$ then has a well-known min-max representation \citep[Theorem~1.4.4]{RobertsonWrightDykstra1988}:
  \[
    \hat f_n(X_i) = \min_{L\in\mathcal{L}, L \ni X_i} \max_{U\in\mathcal{U}, U \ni X_i} \overline{Y_{L\cap U}},
  \]
  where $\overline{Y_{L\cap U}}$ denotes the average value of the elements of $\{Y_1,\ldots,Y_n\} \cap L\cap U$, with the convention that $\overline{Y_{L\cap U}} = 0$ if $\{Y_1,\ldots,Y_n\} \cap L\cap U = \emptyset$. Then 
  \[
    \sup_{x\in[0,1]^d} \hat f_n(x) = \max_{1\leq i\leq n}\hat f_n(X_i) \leq \max_{1\leq i\leq n} Y_i.
  \]
  Since $f_0 = 0$, we have $Y_i = \epsilon_i$, which means that
  \[
    \mathbb{P}\biggl(\sup_{x \in [0,1]^d} \hat f_n(x) \geq 4\log^{1/2} n\biggr) \leq \mathbb{P}\biggl(\max_{1\leq i\leq n}\epsilon_i \geq 4\log^{1/2} n\biggr) \leq n e^{-8\log n} = n^{-7}.
  \]
  The desired result follows by observing that a similar inequality holds for $\inf_{x\in[0,1]^d} \hat f_n(x)$.
\end{proof}

\noindent \textbf{Acknowledgements:} The research of the first author is supported by in part by NSF Grant DMS-1566514.  The research of the second and fourth authors is supported by EPSRC fellowship EP/J017213/1 and a grant from the Leverhulme Trust RG81761.

\end{document}